\documentclass[12pt]{amsart}
\usepackage{color}
\usepackage{lineno} 
\usepackage[all]{xy}
\usepackage{hyperref}
\usepackage{enumerate} 
\definecolor{mygreen}{rgb}{0.0, 0.5, 0.0} 

 \newtheorem{theorem}{Theorem}[section]
 \newtheorem{corollary}[theorem]{Corollary}
 \newtheorem{lemma}[theorem]{Lemma}
 \newtheorem{proposition}[theorem]{Proposition}
 \newtheorem{definition}[theorem]{Definition}

\newtheorem{remark}[theorem]{Remark}

\newtheorem{example}[theorem]{Example}

\newcommand{\defi}{:=}
\newcommand{\N}{{\mathbb N}}             
\newcommand{\Q}{{\mathbb Q}}            
\newcommand{\Z}{{\mathbb Z}}
\newcommand{\R}{{\mathbb R}}

\newcommand{\til}{\widetilde}

\newcommand{\loc}{\mathop\mathrm{loc}}
\newcommand{\tor}{\mathop\mathrm{tor}}
\newcommand{\rank}{\mathop\mathrm{rank}\nolimits}

\newcommand{\comp}{\mathop{\rm comp}\nolimits}

\renewcommand{\hat}{\widehat} 

\newcommand{\ssk}{\smallskip}       
\newcommand{\msk}{\medskip}

\renewcommand{\div}{\mathop{\rm div}\nolimits}

\long\def\alert#1{\parindent2em\smallskip\hbox to\hsize%
{\hskip\parindent\vrule%
\vbox{\advance\hsize-2\parindent\hrule\smallskip\parindent.4\parindent%
\narrower\noindent#1\smallskip\hrule}\vrule\hfill}\smallskip\parindent0pt}

\newcommand{\gp}[1]{\langle#1\rangle}

\newcommand{\gen}[1]{\overline{\gp{#1}}}

\newcommand{\prf}[1]{\Z(#1^\infty)}

\newcommand{\ignore}[1]{} 
\newcounter{alert}0

\newcommand{\im}{in\-duc\-ti\-ve\-ly mo\-no\-the\-tic}

\newcommand{\fg}{finitely generated}

\newcommand{\tqh}{to\-po\-lo\-gi\-cal\-ly qua\-si\-ha\-mil\-ton\-ian}

\newcommand{\stqh}{strong\-ly \tqh}
\newcommand{\stqhg}{\stqh\ group}

\newcommand{\tf}{tor\-sion-free}
\newcommand{\tM}{to\-po\-lo\-gi\-cal\-ly mo\-du\-lar}
\newcommand{\tMg}{\tM\ group}
\newcommand{\locprod}{\prod^{\rm loc}}

\newcommand{\lc}{lo\-cal\-ly com\-pact}
\newcommand{\lca}{\lc\ a\-bel\-ian}

\newcommand{\colim}{{\rm colim}}

\newcommand{\prank}{\hbox{$p$-rank}}

\newcommand{\lead}{\leaders\hbox to 1.5ex{\hss${.}$\hss}\hfill}
\newcommand{\arr}{\hbox to 20pt{\rightarrowfill}}
\newcommand{\larr}{\hbox to 20pt{\leftarrowfill}}

\newcommand{\psyl}[1]{$#1$-Syl\-ow\,}

\newcommand{\socle}{\text{\rm socle}}

\renewcommand{\implies}{{\hbox{$\Rightarrow$}}}

\makeatother

\newcommand{\ZppZp}{(\Z(p^2),p\Z(p^2))^{{\rm loc,}\, \N}}
\makeindex
\begin{document}
\setpagewiselinenumbers

\title[When is the Sum $\ldots$]{When is the Sum of Two Closed Subgroups Closed in a \\ Locally Compact Abelian Group?}

\author{Wolfgang Herfort}

\address{Technische Universit\"at Wien, Wiednerhauptstra\ss e 8-10/E101, Austria,
Email: wolfgang.herfort@tuwien.ac.at}

\author{Karl H. Hofmann}

\address{Fachbereich Mathematik, Technische Universit\"at Darmstadt, Schlossgartenstr. 7, 64289 Darmstadt, Germany, Email: hofmann@mathematik.tu-darmstadt.de}

\author{Francesco G. Russo}

\address{Department of Mathematics and Applied Mathematics, University of Cape Town, South Africa, Email: francescog.russo@yahoo.com}

\begin{abstract}
Locally compact abelian groups are classified in which the sum of any two
closed subgroups is itself closed. 
This amounts to reproving and extending
results by Yu.~N.~Mukhin from 1970. 
Namely we contribute a complete classification of
all totally disconnected \lca\ groups with $X+Y$ 
closed for any closed subgroups $X$ and $Y$.
\end{abstract}
\date{\today}
\keywords{
LCA-group, Dedekind closed subgroup lattice,
MSC[2010] 
}
\maketitle
\newcommand{\dach}{\hbox{$\hat{~}$}}
\newif\ifdraft


\section{Introduction and Main Results}
\label{s:intro-ab}
It was {\sc R. Dedekind}, who
in 1877 proved  the modular law  
for the subgroup lattice of 
a certain abelian group (see \cite{Dedekind77}). 
However, his proof works for {\em any} abelian group.
In 1970 Yu.~N.~Mukhin investigated 
the analogous property for \lca\ groups in \cite{muk2}.
The {\em closed subgroup lattice} $L(G)$ of a topological group 
is its set of closed subgroups endowed with
{\em join} given as $A\vee B\defi \gen{A\cup B}$ 
and {\em meet} as $A\wedge B\defi A\cap B$ for $A$ and
$B$ any closed subgroups. Then $G$ is a {\em \tMg} 
provided the modular law holds for
any closed subgroups $A$, $B$ and $C$ of $G$ with $A$ a subgroup of $C$:
\[A\vee (B\wedge C)=(A\vee B)\wedge C\]
Note that a group is {\em modular} if, and only if, 
its lattice of closed subgroups does not contain
a sublattice isomorphic to $E_5$ (cf. \cite[2.1.2 Theorem]{schmidt}),
i.e., geometrically, a pentagon (see also Remark \ref{rem:not-M} below).

\[\xymatrix@R=2mm@C=3mm{& \bullet\ar@{-}[ld]\ar@{-}[rdd]& \\
            \bullet\ar@{-}[dd]& & \\
                             & & \bullet\ar@{-}[ldd]\\ 
            \bullet\ar@{-}[rd]& & \\
                 &\bullet&}\]

Mukhin, in the same paper, also classifies all \lca\ groups 
$G$ for which $X+Y$ is closed whenever $X$ and $Y$ are closed 
subgroups of $G$ and we will call any \lca\ group
$G$ with this property {\em \stqh}. 
 It follows from the definitions that every \stqhg\ is \tM.

We shall derive his results with slightly different approach, but
essentially the same methods of proof. Our motivation comes from
studying nonabelian \lc\ groups satisfying an analogous property,
see \cite{HHR-stqh-nonab}.

Let us fix some notation. We mostly use the notation from \cite{hofmor}.
The {\em $\Z$-rank} of a discrete abelian 
group $A$ is the {\em torsion free rank}, i.e., the $\Q$-dimension
of $A\otimes_\Z\Q$. When $A$ is \tf\ then
the $\Z$-rank of $A$ is the dimension of
its dual $\hat A$ (see e.g. Theorem 8.22
in \cite{hofmor}). 
We shall call a \lca\ group $A$
{\em periodic} if it is both totally disconnected and the union
of its compact subgroups. 
We shall use additive notation unless
stated differently. For $p$ a prime, an element $x$ 
in a \lca\ group $G$ with $p^kx\to0$ as $k$ tends to infinity,
is called a {\em $p$-element}. As discussed on page 48 in \cite{HHR18}
this definition is equivalent to saying that $\gen x$ is
a pro-$p$ group. 
In a periodic \lca\ group $A$ 
the set of $p$-elements is a 
closed subgroup $A_p$ -- its {\em $p$-primary component} 
(or \psyl p subgroup, 
see \cite[Definitions 8.7]{hofmor}.). 
In a periodic abelian group, 
for a set of primes $\pi$ the {\em $\pi$-primary component} $A_\pi$
of a periodic group $A$ is defined 
to be the subgroup of $A$ topologically generated by all
$p$-primary components with $p\in\pi$. For a periodic group $A$
we denote by $\pi(A)$ the
set of all primes $p$ with $A_p$ not trivial. 
If $\pi(A)=\{p\}$, for a single prime $p$, then $A$ is a {\em $p$-group}.
For a compact group our definition 
agrees with \cite[Definition 8.7]{hofmor}. 
For any fixed prime $p$ the kernel of the map $x\mapsto px$  
is the {\em $p$-socle} of $A$ and will be denoted by $\socle_p(A)$ 
or just by $\socle(A)$ if there is no danger of confusion 
(see \cite[Definition A1.20]{hofmor}.).

A \lca\ group $A$ will be termed {\em \fg} if there is a finite subset
$X$ of $A$ which generates $A$ {\em topologically}, 
i.e., $A=\gen X$. 
We say that a \lca\ $p$-group $A$ has finite \prank\ $n$ if, 
and only if, 
every \fg\ closed subgroup $H$ has a set of topological generators of
cardinality $n$ and, in addition, $A$ contains a \fg\ subgroup
which cannot be generated topologically 
with fewer elements. If $G$ contains \fg\ subgroups of arbitrary large rank
then the \prank\ of $G$ is said to be infinite. 
For a \fg\ compact $p$-group $A$ this definition
agrees with $d(A)$, the minimum cardinality of
a topological set of generators of $A$, 
 as given in \cite[p.~43]{ribes-zalesskii}, i.e.,
$d(A)=\rank_p(A)$. Moreover, by \cite[Proposition 4.3.6]{ribes-zalesskii},
for any closed subgroup $H$ of $A$
one has $d(H)\le d(A)$. Therefore for any finite \prank\ abelian $p$-group
$G$ and closed subgroup $H$ one has $\rank_p(H)\le \rank_p(G)$.

It is a consequence of
\cite[Lemma 3.91]{HHR18} that our definition is equivalent to the one
given by  \v Carin in \cite{Charin66}. 
If $A$ is not topologically \fg, we say that
$A$ has {\em infinite \prank}. A proposal for defining the \prank\
of an arbitrary \lca\ $p$-group has been recently made
in \cite[Section 10]{HHR-comfort}, which got reproduced in \cite[3.10, p.~93]{HHR18}. 
For a prime $p$ the symbols $\Q_p$, $\Z_p$, $\prf p$ denote respectively
the additive group of the field of $p$-adic rationals, its closed subgroup
of $p$-adic integers, and the factor group $\Q_p/\Z_p$ 
(see \cite[p.~28 ff]{hofmor}.). 

The set of elements of finite order of an abelian group $G$ will
be denoted by $\tor(G)$  and we let $\div(G)$ stand for the 
largest divisible subgroup of $G$.
The main properties of $\tor(G)$ and $\div(G)$
are discussed in \cite[Appendix 1]{hofmor}.

Our main results are as follows: 

\begin{theorem}\label{t:mainA}  
The following statements for a \lca\ $p$-group $G$ are equivalent:
\begin{enumerate}[\rm(a)]
\item $G$ is \tM.
\item For $U$ an open compact subgroup exclusively
one of the following holds
  \begin{enumerate}[\rm({b.}\rm 1)]
   \item $U$ has finite \prank. Then $\tor(G)$ is discrete and
        $G/\tor(G)$ has finite \prank.
   \item $U$ has infinite \prank. Then $\div(G)$ is closed,
        $G/U$ and  $\div(G)$ both have finite \prank, 
        and, $G/\div(G)$ is compact. 
  \end{enumerate}
\item $G$ is \stqh.
\end{enumerate}
\end{theorem}

\begin{theorem}\label{t:M-periodic}
A periodic \lca\ group $G$ is \tM\ if, and only if, 
for every prime $p$ the respective $p$-component is \tM.
\end{theorem}

The next two results  exhibit the structure of 
any torsion \stqhg\ and the splitting of the torsion subgroup
in \tMg s. 

\begin{theorem}\label{t:abelian-tor-stqh}
Let $A$ be a \lca\ torsion group.
The following statements are equivalent:
\begin{enumerate}[\rm(a)]
\item The group $A$ is \tM.
\item There
is a partition 
\[\pi(A)=\delta\cup \phi\] 
and all of the following holds:
\begin{enumerate}[\rm({b}.1)]
\item The set of primes $\phi$ is finite and
      \[A_\phi=D_\phi\oplus V_\phi\]
      for $D_\phi$ discrete and divisible 
      and $V_\phi$ compact and open in $G$. 
\item $A_\delta$ is a discrete subgroup of $A$.
\end{enumerate}
\item The group $A$ is \stqh.
\end{enumerate}
\end{theorem}

\begin{theorem}\label{t:mainB} 
Let $G$ be a totally disconnected \lca\ group -- neither discrete
nor periodic. 

Then the following statements are equivalent.
\begin{enumerate}[\rm(a)]
\item $G$ is \tM. 
\item All of the following holds:
  \begin{enumerate}[\rm ({b.}1)]
  \item 
  $T=\tor(G)=\comp(G)$ is open in $G$ and $G/T$ is discrete
  and \tf\ of finite $\Z$-rank.
  \item The torsion subgroup $T=\tor(G)$ is \stqh.
  \end{enumerate}

Moreover, if $N$ is any closed subgroup of $G$, contained in $T$
then $\tor(G/N)=T/N$ and (b) holds for $G/N$ with
$T$ replaced by $T/N$. 

\item $G$ is \stqh.
\end{enumerate}
\end{theorem}
The preceding result corrects 
\cite[Theorem 14.32(b.3)]{HHR18}.

Using Pontryagin duality (see \cite[Chapter 7]{hofmor}) 
we shall deduce a structure theorem for
\lca\ \tM\ groups with nontrivial connected components, see
Theorem \ref{t:M-conn}.


The fact that not every periodic nondiscrete \tM\ group is \stqh,
will be shown in Lemma \ref{l:cyclic-stqh}.
A \lca\ group $A$ is {\em \im} provided
every finite subset of $A$ is contained in a monothetic subgroup
(see Definition \ref{d:img-per} below):

\begin{theorem}\label{t:ab-periodic-stqh}
For a \lca\ periodic group $A$ and open compact 
subgroup $U$ the following statements are equivalent:
\begin{enumerate}[\rm(A)]
\item $A$ is \stqh.
\item There is a partition of $\pi(A)$ into 
      4 disjoint subsets 
      $\delta$, $\gamma$, $\phi$, 
      and $\mu$ and all of the following holds:
          \begin{enumerate}[\rm(i)]
          \item $\delta\defi\{p\in\pi(A):A_p\cap U=\{0\}\}$ 
            and $A_\delta$ is a discrete subgroup of $A$.
          \item $\gamma\defi\{p\in\pi(A):A_p\le U\}$ and $A_\gamma$
            is a profinite subgroup of $A$.
          \item $\phi\defi\{p\in\pi(A)\setminus\{\delta\cup\gamma\}:
             \rank_p(A_p)\ge2\}$. 

            The set $\phi$ is finite and
            for all $p\in\phi$ the \psyl p subgroup $A_p$ is \stqh.
          \item $A_\mu$ is \im.
          \item $A=A_\delta\oplus A_\gamma\oplus A_\phi\oplus A_\mu$ 
            topologically and algebraically.
          \end{enumerate}
\end{enumerate}
\end{theorem}

This result, we feel, is our genuine contribution.
Namely, for periodic \tM\ groups
Theorem 2 in \cite{muk2} and its proof seem not to lead 
to a proof of our description of periodic nondiscrete \stqhg s.
We also correct \cite[Theorem 14.22(B)]{HHR18}, where $\gamma$
and $A_\gamma$ are missing in the decomposition. 

The concluding Section \ref{s:consequences} contains several
consequences.

\section{Preliminaries}
\label{s:prelim-ab}  
In a number of places we shall need a fact about
compact abelian torsion groups, \cite[Corollory 8.9]{hofmor}, 
which we rephrase here:

\begin{proposition}\label{p:cpt-tg}
The following statements about a compact abelian group $G$ are equivalent:
\begin{enumerate}[\rm(a)]
\item $G$ is a torsion group.
\item $G$ is profinite and has finite exponent.
\item $G=\prod_{p\in S}G_p$ is the cartesian product of compact
 $p$-groups of finite exponent for $p$ in a finite set $S$.
\end{enumerate}
\end{proposition}

We have the following observation:

\begin{lemma}\label{l:stqh<tM}
Every \stqhg\ is \tM.
\end{lemma}

\begin{proof}
The equality $(A\vee B)\wedge C=A\vee(B\wedge C)$ for
closed subgroups $A\subseteq C$ and $B$ follows from the containments
\[
\begin{array}{cccccccr}
(A\vee B)\wedge C&=&(A+B)\cap C & 
  \subseteq & A+(B\cap C)&=&A\vee(B\wedge C),\\
A\vee (B\wedge C)&=&A+(B\cap C) & 
  \subseteq & (A+B)\cap C&=&(A\vee B)\wedge C. \\
\end{array}\]
\end{proof}

A result by \v Carin (see \cite[Theorem 5]{Charin66} and a mild
generalisation in \cite{HHR-comfort}) 
will be used frequently.

\begin{proposition} \label{p:finrank} For a \lca\ $p$-group $G$
the following conditions are equivalent:
\begin{enumerate}[\rm(1)]
\item $G$ has finite $p$-rank. 
\item There is a compact open subgroup $U$ such that
      both $\rank_p(U)$ and $\rank_p(G/U)$ are finite.
\item There is a compact open subgroup $U$ such that
      $U\cong \Z_p^m\oplus F$ for a nonnegative integer 
      $m\in\N_0$ and a finite abelian group $F$ and that 
      the $p$-socle 
      $\socle_p(G/U)$ is isomorphic to $\Z(p)^n$
      for some $n\in\N_0$. 
\item 
      There is a natural number $r$ such that every \fg\
      subgroup of $G$ can be generated by at most $r$ elements.
\item There are nonnegative integers $m$, $n$, $k$, and a finite
      $p$-group $F$ such that
      algebraically and topologically
      \[ G\cong \Q_p^m\oplus \prf p^n\oplus \Z_p^k\oplus F\]
\end{enumerate}
\end{proposition}

\begin{lemma}[{\cite[Lemma 3.6]{HHR-comfort}}]\label{l:Rb}
For a \lca\ $p$-group $G$ the following statements are equivalent:
\begin{enumerate}[\rm(a)]
\item $G$ is \fg.
\item $G$ is compact and has finite \prank.
\item There are $m\ge0$ and a finite abelian $p$-group $F$
such that $G$ is algebraically and topologically
 isomorphic to $\Z_p^m\oplus F$.
\end{enumerate}
\end{lemma}

\begin{lemma}\label{l:GU-prank}
Let $G$ be a \lca\ \tf\ $p$-group containing a compact 
open subgroup $U$
of finite \prank. Then \[\rank_p(G)=\rank_p(U).\]
\end{lemma}
\begin{proof} 
By the definition of the \prank\ we need to prove that
every \fg\ subgroup $T$ of $G$ satisfies $d(T)=\rank_p(T)\le r\defi \rank_p(U)$.
Lemma \ref{l:Rb} implies that $T$ is compact and so is $T+U$.
Because of $\rank_p(T)\le\rank_p(T+U)$ 
it will suffice to prove $\rank_p(T+U)\le r$,
i.e., we may assume $U\le T$. Since $T$ is compact and $U$ is an open
subgroup there is $k\ge0$ such that $|T/U|=p^k$.
The homomorphism $\phi:T\to U$ sending $t\in T$ to $p^kt$
is continuous and injective and therefore the compact subgroup
$\phi(T)=p^kT\cong T$ algebraically and topologically.
Deduce from this that 
\[\rank_p(T)=\rank_p(p^kT)\le r=\rank_p(U)\le \rank_p(T),\]
showing $\rank_p(T)=r$, as desired.
\end{proof}

We record a well known fact, see e.g. \cite[2.13 Corollary]{armacost81}:

\begin{lemma}\label{l:Rc}
A \lca\ group $G$ is a $p$-group if, and only if, its dual $\hat G$ is.
\end{lemma}

\begin{remark}\label{rem:tMg}\rm
As has been said above, a lattice is modular if, and only if,
it is $E_5$-free (see \cite[2.1.2 Theorem]{schmidt}). 
The absence of $E_5$ in the closed subgroup 
lattice is inherited by closed subgroups and quotient groups.
Hence closed subgroups and quotient groups of a \tMg\ are \tMg s. 
Moreover, a \lca\ group $G$ is \tM\ if, and only if, its Pontryagin dual
$\hat G$ is \tM\ (the latter fact follows from applying the
Annihilator Mechanism, see \cite[p.~314]{hofmor}). 
\end{remark}     

However, the class of \tMg s fails to be closed under the formation
of strict projective
limits and (local) products  as the following example, due to 
Mukhin  shows (see \cite{muk2} which will be reproduced in
Example \ref{ex:P+S} below). 

\begin{remark}\label{rem:not-M}\rm
Every group $G$ that is not \tM\ must contain closed subgroups $A$, $B$,
and $C$, where $A\subseteq C$,
the meet $B\wedge C$ is a proper subgroup of $A$, and, $C$ is a proper subgroup
of the join $A\vee B$. Then the five closed subgroups
\[A\vee B,\ C,\ A,\ B,\ B\wedge C \tag{$\dagger$}\label{eq:ABC}\]
form a subgroup sublattice of the lattice of closed subgroups of $G$
or equivalently, the five groups in Eq.~(\ref{eq:ABC}) are all pairwise
different and $B\cap C\subseteq A$. 

Indeed, if the closed subgroups $X\subset Z$ and $Y$ do not satisfy the
modular identity then $A\defi X\vee(Y\wedge Z)$, 
$B\defi Y$, and, $C\defi(X\vee Y)\wedge Z$
serve the purpose.
\end{remark}

This observation provides a simple method
for exhibiting important examples of \lca\ groups not \tM.

\begin{example}\label{ex:reals}\rm
Let $G\defi\R$ be the reals and 
fix subgroups $C\defi\Z$, $A\defi 2\Z$, and, $B\defi\sqrt2\Z=
\{z\sqrt2:z\in\Z\}$. 
Then $A\vee B=\R$ by the density of $2\Z+\sqrt2\Z=2(\Z+\frac{\sqrt2}2\Z)$.
Moreover, $B\wedge C=\{0\}$ is contained in $A$.
\end{example}

\begin{example}\label{ex:pK=K}\rm
Let $p$ be a prime and
 $G=Z\oplus K$ be the topological direct sum of a discrete group $Z\cong\Z$
and an infinite compact monothetic group $K$. Suppose that $K=pK$, i.e.,
$K$ is $p$-divisible.

Fix a topological generator $k$ of $K$ and a generator $t$ of $Z$.
Let $C\defi Z$, $A\defi pZ$, and, 
$B\defi\{zt+zk: z\in \Z\}$ and observe that it is
the graph of the homomorphism $f:Z\to K$ sending the
generator $t$ of $Z$ to the generator $k$ of $K$. Hence $B$
is discrete. For proving $G=A\vee B$
observe first that $pK=K$ implies that $\gen{pk}=K$.
Then $A+B$ contains all elements of the form $put+v(t+k)=put+vt+vk$ for
$u$ and $v$ in $\Z$. Select $u\defi 1$ and 
$v\defi-p$ in order to see that $pk\in A+B$.
Therefore $A\vee B=\overline{A+B}$ contains $\gen{pk}=K$
and hence
\begin{equation}\label{eq:AvBC}
(A\vee B)\wedge C=G\cap C=C=Z.
\end{equation}
Suppose there exists an integer $z\in\Z$ such that 
$z(t+k)=zt+zk\in B\cap C$. This implies $zk=0$.
However, $K=\gen k$ is an infinite monothetic group and thus $k$ cannot be a
torsion element. Thus $z=0$ and therefore $B\cap C=\{0\}$, so that
taking Eq.~(\ref{eq:AvBC}) into account, 
\[A\vee (B\cap C)=A\neq C=(A\vee B)\wedge C\]
follows. Thus $G$ is not \tqh.
\end{example}

\begin{example}\label{ex:P+S}\rm
Let $S\defi\Z(p)^{(\N)}$ and $P\defi\Z(p)^{\N}$ 
and form $G\defi S\oplus P$, the 
topological direct sum. Let $\iota:S\to P$ be the canonical dense embedding
of $S$ in $P$ and $K\cong\Z(p)$ a finite subgroup
of $P$ intersecting $\iota(S)$ trivially. Such $K$ can be provided
by the subgroup of all constant maps $\N\to \Z(p)$. 
Define closed subgroups $C\defi S\oplus K$, $A\defi S$, and, 
$B\defi\{(s+\iota(s)):s\in S\}$. Then $B$ is algebraically and
 topologically isomorphic to
the graph of the function $\iota$ and hence a discrete subgroup of $G$.
Then, for $x$ to belong to $B\cap C$ it is necessary and sufficient
that there are $s,s'\in S$ an $k\in K$ with  
\[x=s+\iota(s)=s'+k.\] 
Since $K\cap\iota(S)=\{0\}$ we must have $\iota(s)=s=k=0$.
Hence $B\wedge C=\{0\}$.
Since $A+B=S+\iota(S)$ and $\overline{\iota(S)}=P$
one finds $A\vee B=\overline{A+B}=S+P=G$.
\end{example}

For describing the next examples, and also later, for the 
proof of Theorem \ref{t:M-abelian-p}, we need to recall the notion of
{\em local product} of \lc\ groups.

\begin{definition}\label{d:loc-prod} 
\rm Let $(G_j)_{j\in J}$ be a family of 
\lc\ groups and assume that for each $j\in J$ the group
$G_j$ contains a compact  open subgroup $C_j$. Let $P$ be the   
subgroup of the cartesian product of the $G_j$ containing
exactly those $J$-tuples $(g_j)_{j\in J}$ of elements $g_j\in G_j$
for which the set $\{j\in J: g_j\notin C_j\}$ is finite. Then 
$P$ contains the cartesian product $C\defi\prod_{j\in J} C_j$ which is
a compact topological group with respect to the Tychonoff topology.
The group $P$ has a unique group topology with respect to which $C$
is an open subgroup. Now 
the {\em local product} of the family $((G_j,C_j))_{j\in J}$ is the group $P$
with this topology, and it is denoted by
\[P=\prod_{j\in J}^{\rm loc}(G_j,C_j).\]
\end{definition}

Finally, when $G=\locprod_{i\ge1}(G_i,C_i)$
is a local product and 
$G_i\cong A$ and $C_i\cong B$ algebraically and topologically then
we shall denote $G$ by $(A,B)^{{\rm loc,}\,\N}$. 
\begin{example}\label{ex:p-quadrat}\rm
Let us show that the local product
\[L\defi\ZppZp\]
cannot be \tM.

We are going to show that a closed subgroup of a Hausdorff quotient group
of $L$ is not \tM\ and hence $L$ is not \tM\  by Remark~\ref{rem:tMg}.
Select an infinite subset $I$ of $\N$ with infinite 
complement $J\defi \N\setminus I$. Then there is 
a topological and algebraic isomorphism
\[L\cong L_I\oplus L_J,\]
where $L_I\defi (\Z(p^2),p\Z(p^2))^{{\rm \loc}, I}$ and
$L_J\defi(\Z(p^2),p\Z(p^2))^{{\rm loc}, J}$ are both algebraically
and topologically isomorphic to $L$. The socles $S_I$ and $S_J$ 
of respectively $L_I$ and $L_J$ are compact and open therein and
isomorphic to $\Z(p)^\N$. Since $L_I/S_I\cong \Z(p)^{(\N)}$ 
the subquotient $L_I/S_I\oplus S_J$ of $L$ 
is algebraically and topologically 
isomorphic to $\Z(p)^{(\N)}\oplus \Z(p)^\N$, which is not \tM, 
as has been shown in Example \ref{ex:P+S}.
\end{example}

\begin{lemma}\label{l:pC=C}
Let $C$ be a compact monothetic not torsion group.
Then there is a prime $p$ and a monothetic subgroup $K$ of $C$
with $pK=K$ and $K$ is not torsion.
\end{lemma}

\begin{proof}
If the connected component $C_0$ of $C$ is not trivial we may choose
$K\defi C_0$. Since $C_0$ is divisible, for any prime $p$,
$pK=K$. Since the weight of $C_0$ does not exceed the weight of $C$
infer from \cite[(25.17) Theorem]{hewitt-ross1-book} 
that $C_0$ is monothetic.

Next assume $C_0=\{0\}$. Then $C$ is profinite and abelian and
hence pronilpotent. Making 
use of \cite[Proposition 2.3.8]{ribes-zalesskii}, we deduce
that $C=\prod_pC_p$
is the cartesian product of its \psyl p subgroups. 
If there is a prime $p$ with $C_p=\{0\}$ 
then $K\defi C$ serves the purpose.
Now assume that $C_p\neq\{0\}$ holds for all primes $p$. Select any prime
$p$ and note that by Proposition \ref{p:cpt-tg} the
closed subgroup $K\defi \prod_{q\neq p}C_q$ cannot be torsion.
Certainly $pK=K$.
\end{proof}

\begin{lemma}\label{l:not-tM}
Let $G$ be a \lca\ \tMg. Then 
\begin{enumerate}[\rm(a)]
\item The connected component $G_0$ of $G$ is compact
 and $\comp(G)$ is an open subgroup of $G$.
\item 
If $\comp(G)$ is a proper subgroup of $G$ then $\comp(G)=\tor(G)$.
\item If $U$ is any open compact subgroup then $G/U$ has finite 
  $\Z$-rank.
\end{enumerate}
\end{lemma}

\begin{proof}
(a)
Since $G$ is \tM\ so is, by Remark \ref{rem:tMg}, 
the connected component $G_0$. By the Vector Splitting Theorem (see 
\cite[Theorem 7.57]{hofmor}) there is $n\ge0$ and a compact
connected subgroup $K$ such that
\[G_0=\R^n\oplus K.\]
If $G_0$ were not compact then $n>0$ and hence there is a closed
subgroup $R\cong\R$ of $G_0$ which must be \tM,
contradicting the findings in Example \ref{ex:reals}.
Hence $G_0=K$ is compact.     
The factor group $G/G_0$ is totally disconnected and thus
contains a compact open subgroup, say $C$. 
The latter gives rise to an open
compact subgroup of $G$. 
Therefore $\comp(G)$ is open.

\msk

(b) 
Since $\comp(G)<G$, (a) implies that
the factor group $G/\comp(G)$ is discrete
and \tf. Therefore one can find a discrete subgroup $Z\cong\Z$ of $G$.
Suppose $G$ to contain an element $c$ with $C\defi\gen c$ 
compact and not torsion.
Then $C$ is an infinite monothetic subgroup.
Lemma \ref{l:pC=C} provides a prime $p$ and a monothetic infinite
subgroup $K$ of $C$ with $K=pK$. Remark \ref{rem:tMg} shows that
the closed subgroup $Z\oplus K$ must be \tM. This leads to a contradiction
in light of Example \ref{ex:pK=K}.

\msk

(c)
If, for some open compact subgroup $U$ of $G$, the factor group
$G/U$ has infinite $\Z$-rank then $G$ contains a closed subgroup 
$S\cong \Z^{(\N)}\oplus U$. By (b) $U$ is a compact torsion group
and by Proposition \ref{p:cpt-tg} it
is the cartesian product $U=\prod_{p\in S}U_p$ of compact finite
exponent $p$-groups for a finite set $S$ of primes. 
Since $U$ is assumed to be infinite (else $G$ would be discrete)
there is $p\in S$ with $U_p/pU_p$ infinite. Since $G$ is
by assumption \tM, so is $R\defi\Z^{(\N)}\oplus U_p/pU_p$. 
Let $(V_i)_{i\in\N}$ be a properly descending sequence of open subgroups of
$R$ all contained in $U_p/pU_p$ and let 
$V\defi\bigcap_{i\ge1}V_i$ denote the intersection.
Then $(U_p/pU_p)/V$ is first countable and has exponent $p$.
Therefore $(U_p/pU_p)/V\cong \Z(p)^\N$ and hence   
topologically and algebraically 
\[R/V\cong \Z^{(\N)}\oplus \Z(p)^{\N}.\]
Letting $A=\Z^{(\N)}$ denote the first direct summand
we may factor $pA$ and obtain the \tM\ $p$-group 
\[(R/V)/pA\cong \Z(p)^{(\N)}\oplus\Z(p)^{\N}\]
contradicting our findings in Example \ref{ex:P+S}.
\end{proof}

If a periodic group is the topological direct sum
of groups $G$ and $H$ and
$\pi(G)\cap\pi(H)=\emptyset$ it will be enough to ensure that each factor
is \stqh, in order to prove that $G\oplus H$ is \stqh. 

\begin{lemma}\label{l:coprime}
If $G$ and $H$ are both periodic \stqhg s and $\pi(G)\cap\pi(H)=\emptyset$ 
then their topological direct sum $G\oplus H$ is a \stqhg.
\end{lemma}

\begin{proof}
Put $\pi\defi \pi(G)$ and $\sigma\defi \pi(H)$. 
For closed subgroups $X$ and $Y$
there is a corresponding decomposition
\[X=X_\pi\oplus X_\sigma, \ \ Y= Y_\pi\oplus Y_\sigma.\]
Then $X_\pi+Y_\pi$ and $Y_\sigma+Y_\sigma$ are both closed subgroups in
respectively $G$ and $H$ by our assumptions. Hence
\[ X+Y=(X_\pi+ Y_\pi)\oplus(X_\sigma+Y_\sigma) \]
is a closed subgroup of $G\oplus H$.
\end{proof}

The following fact has already been observed in \cite[Remark 2]{muk5}.

\begin{lemma}\label{l:cyclic-stqh}
Let $I$ be a nonempty index set and select 
for every $i\in I$ a prime $p_i$. Set
\[A\defi \bigoplus_{i\in I}\Z(p_i)\times \prod_{i\in I}\Z(p_i).\]
Then $A$ is \stqh\ if, and only if, $I$ is finite.
\end{lemma}

\begin{proof}
Suppose that $A$ is \stqh.
Let $c_i$ be a topological generator of $\Z(p_i)$ 
in the profinite factor
\[C\defi \prod_{i\in I}\Z(p_i)\]
of $A$
 and $b_i$ for $\Z(p_i)$ in the discrete  factor
 \[B\defi  \bigoplus_{i\in I}\Z(p_i)\]
  of $A$. Then
\[A=B\oplus C\]
where $C$ is a compact open subgroup of $A$. For $i\in I$ 
set $a_i\defi b_i+c_i$ and define  $X\defi\gen{a_i:i\in I}$. 
View $X$ as the graph of the
obvious injection $\iota\colon B\to C$ in $B\times C$.
A graph of any continuous function is
always homeomorphic to the domain and 
therefore $X$ is a discrete subgroup of $G$.
Set $Y\defi B$. Then 
\[X+Y=\gp{c_i:i\in I}+\gp{b_i: i\in I}=B+\iota(B)\] 
is dense in $B\times C$ 
and is therefore closed if, and only if, $I$ is finite.
\end{proof}

\begin{remark}\rm\label{r:tqh-stqh}
If $I$ is infinite countable and 
the primes $p_i$ are pairwise different 
then we will show later, in Theorem \ref{t:M-periodic},
that $A$ is \tM\ and not \stqh. 
Note that $A$ is the projective limit with
compact kernels of discrete \stqhg s.
\end{remark}

The {\em good} 
properties of the class of \stqhg s are the 
following ones. 

\begin{proposition}\label{p:stqh-class}
Let $\mathfrak X$ be either the class of \tMg s or of \stqhg s.
Then $\mathfrak X$  is closed under 
\begin{enumerate}[\rm (a)]
\item passing to closed subgroups; and
\item passing to factor groups modulo closed normal subgroups.
\end{enumerate}
\end{proposition}

For its proof we first establish an elementary fact.

\begin{lemma}\label{l:quotient-closed}
Let $G$ be a topological group and $N$ a closed normal subgroup.
Then any subgroup $S$ containing $N$ is closed in $G$ if and only if
$S/N$ is a closed subgroup of $G/N$.
\end{lemma}

\begin{proof} 
Let $\phi:G\to G/N$ denote the quotient map. Then $\phi(S)\subseteq G/N$
is closed if, and only if, $S+N=\phi^{-1}(\phi(S))$ is closed in $G$.
\end{proof}

\begin{proof}[Proof of Proposition \ref{p:stqh-class}]
When $\mathfrak X$ is the class of all \tMg s then 
(a) and (b) follow from the fact that the lattice
of closed subgroups must not contain the graph $E_5$.

\msk 

We turn to $\mathfrak X$ being the class of \stqhg s.
Let $G\in\mathfrak X$ and $L$ a closed subgroup. Then the product
of any two closed subgroups of $L$ is a closed subgroup of $G$ and hence
of $L$. Thus $L$ is \stqh.

That $G/N$ is \stqh\ follows from Lem\-ma \ref{l:quotient-closed}.
\end{proof}

A fact about certain $p$-groups of exponent $p^2$ and the
local product 

\[L\defi\ZppZp \tag{*}\label{eq:p2p}\]

will be needed.

From Example \ref{ex:p-quadrat} it should be clear that
we are looking for information which secures that a locally compact
abelian $p$-group may have a quotient which contains a subgroup 
isomorphic to $L$ in Eq.~(\ref{eq:p2p}).


The following discussion serves this purpose

The group $L$ has two significant components, namely,
\[P\defi p\Z(p^2)^\N, \]
the socle of $L$, a compact open characteristic subgroup, and
\[F\defi\Z(p^2)^{(\N)},\]
a noncharacteristic dense countable subgroup
such that 
\[L\defi F+P, \mbox{ and } p\.F=F\cap P, \mbox{ dense in }P.\]
We observe that we have a basis of compact open zero neighborhoods
\[P_m\defi p\Z(p^2)^{\{n\in\N:m\le n\}},\ m\in\N\]
in $L$, and an ascending union of discrete 
finite subgroups $F_0=\{0\}$ and
\[F_m\defi\Z(p^2)^{\{n\in\N: n<m\}},\ m\in\N\]
such that 
\begin{align*}
F=\bigcup_{m\in\N}F_m \mbox{ and} \\
L=\bigcup_{m\in\N}(F_m\oplus P_m)    &=\colim_{m\in\N}F_m\oplus P_m,\\ 
  (\forall m\in\N)\ \ (p\.F_m\oplus P_m) &= P.\hfill\\
\end{align*}
Finding a copy of $L$ in a $p$-group $G$ now amounts to finding
cyclic subgroups $Z_k$ isomorphic $\Z(p^2)$ matching these configurations.

\begin{lemma}\label{l:p2p}
Assume that a locally compact abelian $p$-group $G$
has a descending basis of $0$--neighborhoods of compact open subgroups 
$V_1\supseteq V_2\supseteq V_3\supseteq\cdots$ 
and  a family $\{Z_k:k\in\N\}$ of subgroups with
isomorphisms  $\zeta_k\colon \Z(p^2)\to Z_k$
 satisfying the following conditions

\begin{enumerate}[\hspace*{1em}\rm(a)]
\item $V_1$ has exponent $p$.
\item The sum $Z_1+\cdots+Z_m+V_{m+1}$ is direct (algebraically and
topologically) for $m=1,2,\dots$.

\item $p\.Z_k\subseteq V_k$ for $k=1,2,\dots$.
\end{enumerate}

\noindent Set $F_G\defi\sum_{k\in\N}Z_k$, and $L_G\defi\overline{F_G}$,
further $P_G\defi\overline{p\.F_G}\subseteq V_1$.
Then $L_G$ is equal to the sum $F_G+P_G$ and all of the following holds: 
\begin{enumerate}[\rm(i)]
\item there is an algebraic isomorphism 
$\eta_F\colon \Z(p^2)^{(\N)}\to F_G$ such that the restriction
to the $k$-th summand is $\zeta_k$.
\item There is an isomorphism of compact groups
$\eta_P\colon p\Z(p^2)^\N \to P_G$ 
such that the restriction to the $k$-th factor is $\zeta_k|p\Z(p^2)$.
\item There is an isomorphism of topological groups
$\eta_L\colon L\to L_G$.
\end{enumerate}
\end{lemma}

\begin{proof}
There is no loss of generality to assume $G=L_G$.
For each $k$ we have an isomorphism $\zeta_k\colon\Z(p^2)\to Z_k$
by the definition of $Z_k$. Conclusion (i) follows from Assumption (b).
Since $F_G$ is dense in $L_G=G$ we have 
that $F_G\cap V_1$ is dense in the open set $V_1$,
i.e., $\overline{F_G\cap V_1}=V_1$. 
Since $V_1$ has exponent $p$ the equation $F_G\cap V_1=pF_G$
follows from (c). 
Then 
\begin{equation}\label{eq:FGV1}
V_1=\overline{F_G\cap V_1}=\overline{pF_G}=P_G \end{equation}
is open in $G$. As a consequence
by the density of $F_G$ in $L_G$ we have 
\begin{equation}G=F_G+P_G.\label{eq:FGPG}\end{equation} 
Furthermore,
\[(\forall k\in\N)\, Z_k\cap P_G=Z_k\cap V_1\cap F_G=Z_k\cap pF_G=pZ_k.\]

We now prove (ii). 

Let us set $F_{G,k}\defi \bigoplus_{1\le j\le k}Z_j$. 
We claim that for any $k\ge1$
\[P_G=pF_{G,k}\oplus V_{k+1}.\] 
Passing in (c) on both sides to the union and noting
that $(V_l)_{l\ge1}$ is decreasing one obtains 
\[\bigcup_{l\ge1}pZ_l=
\left(\bigcup_{1\le l\le k}pZ_l\right)\cup\left(\bigcup_{l\ge k+1}pZ_k\right)
\subseteq 
\left(\bigcup_{1\le l\le k}pZ_k\right)\cup \left(\bigcup_{l\ge k+1}V_l\right)
=
\left(\bigcup_{1\le l\le k}pZ_k\right)\cup V_{k+1}.\]
Observing that the set on the left hand side generates $pF_G$
and, taking (b) into account, one arrives at
$pF_G\le pF_{G,k}\oplus V_{k+1}$.
 
Hence $P_G=\overline{pF_{G}} \le pF_{G,k}\oplus V_{k+1}$.
For proving the converse containment, we take Eq.~(\ref{eq:FGV1}) into
account and observe   
\[pF_{G,k}\oplus V_{k+1}\le pF_G+V_1=P_G.\] 
Thus,
\[P_G=pF_{G,k}\oplus V_{k+1},\]
that is, there is a projection $p_k\colon P_G\to pF_{G,k}$,
and for each $k\ge2$ there is a canonical
projection $\phi_k:pF_{G,k}\to pF_{G,k-1}$ with kernel $pZ_k$.
So $(pF_{G,k},\phi_k)_{k\in\N}$ forms 
an inverse system with projective limit 
$\varprojlim_k pF_{G,k}= \prod_{k\ge1}pZ_k$. 
Let us prove the equality \[\phi_k\circ p_k=p_{k-1}, \ \ k\ge2.\]
As $P_G=pF_{G,k}\oplus V_{k+1}$ we may decompose $x\in P_G$
as $x=pf+v$ for $f\in F_{G,k}=\bigoplus_{1\le j\le k}Z_j$ and $v\in V_{k+1}$.
Therefore
\[(\phi_k\circ p_k)(pf+v)=\phi_k(pf).\]
Decomposing $f=f_1+z_k$ for some $f_1\in F_{G,k-1}$ and $z_k\in Z_k\le V_k$,
the expression on the right yields 
\[\phi_k(pf)=\phi_k(pf_1+pz_k)=pf_1=p_{k-1}(pf_{k-1}+pz_k+v)=p_{k-1}(x),\] 
as needed. 
Therefore, by the universal property of the limit, 
there is a unique morphism 
\[\phi\colon P_G\to \varprojlim_kpF_{G,k}=\prod_k pZ_k. \]
Since all morphisms $(p_k)_{k\ge1}$ are surjective, so is
$\phi$ and and since these morphisms separate the points,
$\phi$ is an isomorphism of compact groups.
By the definition of $Z_k$ we have an isomorphism 
$\alpha\colon p\Z(p^2)^\N\to \prod_{k\ge1}pZ_k$ so that the
restriction and corestriction to the $k$-th factor agrees
with $\zeta_k|p\Z(p^2):p\Z(p^2)\to pZ_k$. 
Thus, as has been claimed, 
 $\eta_P=\phi^{-1}\circ\alpha\colon p\Z(p^2)^\N\to P_G$
is an isomorphism 
mapping the $k$-th factor of $p\Z(p^2)^\N$ to $pZ_k\subseteq P_G$.

For a proof of (iii) denote $\Z(p^2)^{(\N)}$ by $F$ and note
that it is a free $\Z(p^2)$-module. Therefore there is a homomorphism
$\eta_F:F\to \sum_{k\ge1}Z_k$ which restricts to $\zeta_k$ on the
$k$-th direct summand of $F$. Take an element 
$z=(z_n)_{n\in\N}\in \Z(p^2)^{(\N)} \cap p\Z(p^2)$. 
Then $\eta_F(z)=\sum_{n\in\N} \zeta_n(z_n)$ by (i) in view of the
definition of a direct sum. Therefore the restrictions of respectively $\eta_F$
and $\eta_P$ from (ii)
to $pF=(p\Z(p^2))^{(\N)}$ agree. 
Setting $P\defi (p\Z(p^2))^{\N}$ one observes that
$\eta_F$ and $\eta_P$ agree on $F\cap P$ and thus define
a unique algebraic morphism $\eta\colon L=F+P\to F_G+P_G=G$, 
where $F_G$ and $P_G$ are as in Eq.~(\ref{eq:FGPG}). 

Since $\eta$ agrees on the open subgroup $P$ with the continuous and
open map $\eta_P$ it is continuous and open. 
Since $\eta_F$ is an isomorphism,
$F_G$ is in the image of $\eta$. Similarly, $P_G$ is in the image of
$\eta$ as well. Hence $\eta$ is surjective. If $\eta(z)=0$, and 
$z\in F$, then $0=\eta(z)=\eta_F(z)$ implies $z=0$ since $\eta_F$
is injective. If $z\in P$ then $0=\eta_F(z)=\eta_P(z)$ and
by (ii) we must have $z=0$.  
If $z=f+v$ for $f\in F$ and $v\in P$ 
then $0=\eta(z)=\eta(f)+\eta(v)$ implies 
$0=p\eta(f)=\eta(pf)$ so that from $pf\in P$ and $\eta\vert_{P}=\eta_F$
we may deduce $pf=0$. Therefore $f$ itself belongs to $P$
and thus $\eta(f+\nu)=0$ implies $z=f+\nu=0$.
Hence $\eta$ 
is injective and thus is an isomorphism of topological groups.
This completes the proof.
\end{proof}

\begin{proposition}\label{p:findim-neu}
Let $U$ be a closed totally disconnected subgroup of 
a compact connected $n$-dimensional abelian group $G$. Then,
for every $p\in\pi(G)$, the \prank\ of the \psyl p subgroup 
$U_p$ of $U$ is bounded by $n$. 

In particular, every subgroup of $G$ of finite exponent is finite. 
\end{proposition}

\begin{proof}
By \cite[Corollary 8.24(iv)]{hofmor} we have $\dim(G/U)=\dim(G)=n$.
Observing that $G$ is connected, duality applied to the
exact sequence
\[\{0\}\to U\to G\to G/U\to \{0\}\]
renders an exact sequence
\[\{0\}\to \widehat{G/U}\to \hat G\to \hat U\to\{0\}\]
where the second and the third term are \tf\ groups of $\Z$-rank $n$
and $\hat U$ is a discrete torsion group.
Since $\hat G$ is a subgroup of $\Q^n$ and $\widehat{G/U}$ must contain
a subgroup $L\cong\Z^n$ it follows that $\hat U$ must be a subgroup
of a quotient of
\[\Q^n/L\cong \bigoplus_p\prf p^n.\]
Therefore the \prank\ of $\Q^n/L$ does not exceed $n$. Hence
$\hat U$ has \prank\ not exceeding $n$. 
Recalling from \cite[Definition 3.1]{HHR-comfort} that the \prank\ of
$U_p$ is precisely the \prank\ of the socle  of $\hat U_p$ we arrive
at $\rank_p(U_p)\le n$, as needed.

\msk

For proving the second statement, let $E$ be a subgroup of $G$ of 
finite exponent, say $e$. Then its closure $\overline E$ also has
exponent $e$ and thus there is no loss of generality to assume that
$E$ is closed and hence compact. 
Therefore Proposition \ref{p:cpt-tg} implies that 
$E=\prod_{p\in S}E_p$ for a finite set of $S$ of primes and,
moreover, $E_p$ has finite exponent. 
By the first part of the proof we know that $\rank_p(E_p)$ must
be finite, and therefore, $E_p$ is finite for every $p\in S$, and
so is $E$.
\end{proof}

\begin{corollary}\label{c:findim-neu}
Suppose $G$ is a \lca\ group with compact finite dimensional connected 
component $G_0$ and suppose that $\tor(G/G_0)$ is a discrete subgroup
of $G/G_0$ and has finite exponent, say $e$.
Then $E\defi\gp{x\in G:ex=0}$ is a discrete subgroup of $G$.
\end{corollary}

\begin{proof}
By \cite[(5.23) Lemma]{hewitt-ross1-book} the isomorphism
$(E+G_0)/G_0\to E/(E\cap G_0)$ maps open sets to open sets 
and since $(E+G_0)/G_0$
is discrete we may conclude that $E/(E\cap G_0)$ is discrete.
Proposition \ref{p:findim-neu} shows that $F\defi E\cap G_0$ is finite.
Since $E$ is a closed totally disconnected subgroup of $G$ it contains
a compact open subgroup $V$ with $V\cap F=\{0\}$. Therefore the
discrete compact and hence finite group $(V+G_0)/G_0$
maps onto $V/(V\cap G_0)$ by the above map showing that $V$ itself is
finite. Thus $E$ is a discrete subgroup of $G$. 
\end{proof} 

\begin{lemma}\label{l:XYU} 
Let $G$ be a \lca\ group and $X$ and $Y$ be closed subgroups.
Let $U_X$ and $U_Y$ be compact subgroups of respectively $X$ and $Y$.
Then $\til X\defi X+U_Y$ and $\til Y\defi Y+U_X$ 
are closed and $\til X+\til Y=X+Y$. 
Letting $K\defi U_X+U_Y$, the subgroup 
$X+Y$ is closed in $G$
if, and only if, $\til X/K+\til Y/K$ is closed in $G/K$.

If there is a compact subgroup $U$ with $U_X=X\cap U$ and 
$U_Y=Y\cap U$ then we have $\til X\cap U=\til Y\cap U=K$. If, in
addition, $U$ is open, then $X/(X\cap U)$ is a discrete
subgroup of $G/(X\cap U)$.
\end{lemma}

\begin{proof}
Since $K=U_X+U_Y$ is the sum of
compact subgroups of $G$ it is compact. Certainly 
\[\til X+\til Y=X+U_Y+Y+U_X=X+U_X+Y+U_Y=X+Y.\]
Now the result follows from Lemma \ref{l:quotient-closed}.

For proving the second statement, we only show that $\til X\cap U=K$,
as the equality $\til Y\cap U=K$ can be proved along the same lines.
By construction, $K\le \til X\cap U$. Now fix $z\in \til X\cap U$.
Then there are $x\in X$ and $y\in Y\cap U=Y_U$ with $z=x+y$. 
Since $y\in K\le U$ and $z\in U$ we can conclude $x\in X\cap U=X_U$. Hence
$z=x+y\in X_U +K=K$. Thus $\til X\cap U=K$.

Suppose that $U$ is compact open. Then $(X+U)/U$ is discrete and
\cite[(5.32) Theorem]{hewitt-ross1-book} 
implies that the isomorphism $(X+U)/U\to X/(X\cap U)$ maps open
subsets of $(X+U)/U$ to open subsets of $X/(X\cap U)$. 
Hence $X/(X\cap U)$ is a discrete subgroup of $G/(X\cap U)$. 
\end{proof}

\begin{lemma}\label{l:XK}
Let $G$ be a \lca\ group.
Suppose $K$ is a compact subgroup of $G$ and $X$ a closed subgroup of $G$. 
Then $X+K$ is closed in $G$. Furthermore, $X$ is compact
if, and only if, $(X+K)/K$ is compact.
\end{lemma}

\begin{proof}
Certainly $X+K$ is closed and if $X$ is compact, so is $(X+K)/K$.
On the other hand $X+K$ is \lca, so if $(X+K)/K$ is compact then so is $X+K$ by
\cite[(5.25) Theorem]{hewitt-ross1-book} and therefore $X$ is compact.
\end{proof}

\begin{lemma}\label{l:phi-kappa}
Let $H$ be a closed subgroup of a \lca\ group $A$ satisfying the following
premises:
\begin{enumerate}[\rm(a)]
\item $A$ is periodic and $\phi\defi\pi(A)$ is finite.
\item $E\defi\tor(A)$ is discrete and has finite exponent, say $e$.
\item $A/E$ is \fg.
\end{enumerate}
Then, algebraically and topologically, 
$H=Z\oplus E_H$ for $Z$ \tf\ and \fg\ and $E_H=\tor(H)$.
\end{lemma}
\begin{proof}
We first claim that $A=X\oplus E$ for $X$ a \fg\ \tf\ subgroup of $A$.
Since $\phi$ is finite and, algebraically and topologically,
 $A=\bigoplus_{p\in\phi}A_p$ for $A_p$ the $p$-primary
subgroup of $A$, it will suffice to prove the claim under the additional
assumption that $\phi=\{p\}$ and thus $A$ is a \lca\ $p$-group.

According to (c), $A/E$ is \fg. Hence Lemma \ref{l:Rb} 
implies for some $r\ge0$ the topological isomorphism
$A/E\cong\Z_p^r$. Lifting topological
generators of $A/E$ to $A$ gives rise to a closed subgroup $X\cong\Z_p^r$
of $A$ with $A=X\oplus E$. The claim holds.

\msk

As during the proof of the claim, 
for establishing the statements about $H$,
we may assume that $H$ is a $p$-group for some prime $p$.
Certainly $E_H=\tor(H)=H\cap \tor(A)$ is discrete. 
The continuous epimorphism
$\phi:A\to A/E\cong\Z_p^r$ restricts to a map $\chi:H\to A/E$ with kernel
$E_H=\tor(H)$. The induced homomorphism 
$\bar\chi$ from the compact group $H/E_H$
to $A/E$ is continuous and renders a compact image, say $L$, in $A/E$. 
It follows
from the compactness of $H/E_H$ and $A/E$ that 
$\bar\chi:H/E_H\to L$ is an isomorphism of topological groups
 and hence $H/E_H$ must
be \fg. Therefore, replacing in the above claim $A$ by $H$, we can deduce
the existence of a \fg\ subgroup $Z$ of $H$ with $H=Z\oplus E_H$ 
algebraically and topologically.
\end{proof}

\section{Proving the Main Results}
We shall proceed in three subsections, dealing first with $p$-groups,
then with totally disconnected ones, and finally with groups having
nontrivial connected components.

\subsection{$p$-Groups}
\label{ss:Mg-p-case}
Let us first describe the structure
of an abelian \tM\ $p$-group.

\begin{theorem}[Mukhin, see \cite{muk2}]\label{t:M-abelian-p}
 A \lca\ \tM\ $p$-group $G$ satisfies one of the following conditions:
\begin{enumerate}[\rm(a)]
\item $G$ contains an open compact subgroup of finite \prank.
 Then the torsion subgroup $T=\tor(G)$ of\, $G$ is discrete and $G/T$
 has finite \prank.
\item There is an open compact subgroup $U$ of $G$ with infinite \prank.
 Then $G/U$ has finite \prank\ and $G$ 
 contains a closed subgroup $D$ of finite $p$-rank 
 with compact factor group $G/D$. In particular, $D$ can be taken to be
 $\div(G)$. 
\end{enumerate}
\end{theorem}

\begin{proof}
Suppose first that the premise of (a) is valid, i.e., $\rank_p(U)$ is
finite for some open compact subgroup $U$ of $G$. Then, taking 
Proposition~\ref{p:finrank} into account, we may
replace $U$ by a suitable of its open 
subgroups and achieve that $U$ is \tf.
Therefore $\tor(G)$ must be a discrete subgroup. 
As the \tf\ group $G/\tor(G)$ contains 
the open subgroup $(U+\tor(G))/\tor(G)$
and the latter has finite \prank, Lemma \ref{l:GU-prank} implies that
$G/T=G/\tor(G)$ has finite \prank. 

\msk

Let us assume the premise of (b) now. 
Suppose, by way of contradiction, the \prank\ of $G/U$ to be infinite.
We shall derive a contradiction from this by showing that a local
product isomorphic to the one in Eq.~(\ref{eq:p2p})
can be manufactured to be a factor group of a closed subgroup of $G$,
being \tM\ by Remark \ref{rem:tMg}, 
and then refer to Example \ref{ex:p-quadrat}.

\msk

{\em Claim 1: One can assume $U$ to have exponent $p$.}

\msk

With $G$ also $G/pU$ is \tM. Consider instead of $G$ and $U$
the factor group $G/pU$ and its open compact subgroup $U/pU$
which still has infinite \prank.

\msk

{\em Claim 2: One can assume $U$ to be first countable and hence metric. 
 Furthermore, one can arrange $G/U\cong\Z(p)^{(\N)}$ and $U\cong\Z(p)^\N$.

Moreover, every open subgroup $V$ of $U$ is isomorphic to $U$. In particular, $V$ is metrizable, has infinite \prank, and, has exponent $p$.}

\msk

There is a strictly decreasing sequence $(V_k)_{k\ge1}$ of open
subgroups of $U$. Letting $V\defi\bigcap_{k\ge1}V_k$
we pass from $(G,U)$ to $(G/V,U/V)$. Then $G/V$ is \tM\ and
$U/V$ is first countable and infinite of exponent $p$. 
Therefore $U/V$ has infinite \prank\
and is thus isomorphic to $\Z(p)^{\N}$. 
By replacing $G$ with the inverse image of $\socle(G/U)$ under
projection $G\to G/U$ we achieve $G/U\cong\Z(p)^{(\N)}$. 

The ``moreover'' statement follows from $U\cong\Z(p)^{(\N)}$
and \cite[Theorem 4.3.8]{ribes-zalesskii}.

\msk

{\em Claim 3: One can assume that $\overline{pG\cap U}$ is
open in $G$.  Moreover, for any open subgroup $V$ of $U$ the 
intersection $pG\cap V\neq\{0\}$.} 

\msk

If $T\defi\overline{pG\cap U}$ is {\em not} open in $G$ it must have
infinite index in $U$. By Claim 2 the group $G$ has exponent $p^2$
and $pG\le U$. Therefore the factor group $G/T$ 
has exponent $p$ and, as $G/U$ is infinite, so is
$(G/T)/(U/T)$. Since $G/T$ may be considered a GF$(p)$-vector space
the open subgroup $U/T$ admits a complement, say $S$.
Thus algebraically and topologically
\[G/T\cong S\oplus U/T.\]
Select in $G/T$ a countable subgroup $\Sigma_p$ of $S$ 
isomorphic to $\cong\Z(p)^{(\N)}$. 
Since $U/T$ is a compact group of exponent $p$ (by Claim 1) and
metrizable (by Claim 2), it follows that it
is topologically isomorphic to $\Z(p)^\N$. Hence it turns out 
that $\Sigma_p\oplus U/T$ is a factor group of a closed subgroup of $G$
and therefore is \tM\ by Remark \ref{rem:tMg}.
This contradicts the finding in Example \ref{ex:P+S}.

\msk

For proving the second statement, suppose, by way of contradiction
that $pG\cap V=\{0\}$ for some open subgroup $V$ of $U$.
The (purely algebraic) isomorphism
\[pG\cong pG/(pG\cap V)\cong (pG+V)/V\]
and the fact that $V$ and $pG+V$ are open subgroups of $U$ implies
that $pG$ must be finite. But then the open subgroup 
$\overline{pG\cap U}$ would be finite and $G$ would be discrete, a
contradiction.

\msk

{\em Claim 4: There is a sequence $(F_k,V_k)_{k\ge1}$ 
of pairs of compact subgroups of $G$ where
for all $k\ge1$
\begin{enumerate}[(1)]
\item $F_k$ is finite, $V_k$ is open, and, $F_k\cap V_{k+1}=\{0\}$; and
\item $F_k\cap V_k=\gp{px_k}$ for some nontrivial element $x_k\in G$; and
\item $F_k\subseteq F_{k+1}$ and $V_{k+1}\subseteq V_k$
and $\bigcap_{k\ge1}V_k=\{0\}$.
\end{enumerate}
}

\msk

We proceed by induction on $k$ and recall that for all $k\ge1$
the subgroups $V_k$ of $U$ will be separable of infinite \prank, and,
have exponent $p$ (see Claim 2).
For $k=1$ let $V_1$ be the open subgroup $\overline{pG\cap U}$ (see Claim 3).
Then there exists $x_1\in G$ of order $p^2$ with $px_1\in V_1$ and
we set $F_1\defi\gp{x_1}$.

Suppose $(F_i,V_i)$ have been found for $1\le i\le k$.
Then $F_k$ is finite and hence there exists an open subgroup $V_{k+1}$
contained in $V_k$ with $F_k\cap V_{k+1}=\{0\}$.  
By Claim 3 the intersection $pG\cap V_{k+1}\neq\{0\}$. 
Therefore one can find
$x_{k+1}\in G$ of order $p^2$ with $px_{k+1}\in V_{k+1}$.
Set $F_{k+1}\defi F_k\oplus \gp{x_{k+1}}$. 

(1) and the first statement of (3) 
are now clear from the construction.
For proving (2) suppose $x\in F_{k+1}\cap V_{k+1}$.
Then $x=f_k+\lambda x_{k+1}$ for some $f_k\in F_k$ and $0\le\lambda
\le p^2-1$.  
Since $x\in V_{k+1}$ we must have $0=px=pf_k+p\lambda x_{k+1}$
and $pf_k=-\lambda px_{k+1}\in F_k\cap V_{k+1}=\{0\}$ (by the construction
of $V_{k+1}$).
Hence $\lambda=p\mu$ for some $0\le\mu\le p-1$.  
This implies $x=\lambda x_{k+1}=\mu px_{k+1}\in\gp{px_{k+1}}$.

Selecting at each step $V_i$ small enough 
one can achieve the second statement in (3), namely
$\bigcap_{i\ge1}V_i=\{0\}$.

\msk

Setting in Claim 4 for all $k\in\N$ respectively $Z_k\defi\gp{x_k}$
and $F\defi\gp{Z_k:k\ge1}$ shows that the assumptions of
Lemma \ref{l:p2p} hold. Therefore there is a closed
subgroup $L$ of $G$ topologically and algebraically isomorphic
to the group in Eq.~(\ref{eq:p2p}).
We have reached a contradiction and therefore the \prank\ of
$G/U$ must be finite.

\msk

For proving the remaining assertions of (b)
let us return to the original meaning of $G$ and its open compact
subgroup $U$ of infinite \prank. By what we just proved,
the factor group $G/U$ has finite \prank.
Thus, according to Proposition \ref{p:finrank},
$G/U\cong \prf p^m\oplus F$, for some $m\ge0$ and finite group $F$.
Replacing $U$ by the preimage of $F$ in $G$ allows to have $F=\{0\}$.
Lemma \ref{l:Rc} implies that $\hat G$ is \tM\ and 
duality theory applied to the short exact sequence $U\to G\to\prf p^m$ 
implies that $\hat G$ must contain an open compact subgroup $\cong\Z_p^m$.

Hence $\hat G$ satisfies the premise of (a) and therefore $\tor(\hat G)$
is a discrete subgroup of $\hat G$ with $\hat G/\tor(\hat G)$
being \tf\ of finite \prank. 
Therefore Proposition \ref{p:finrank}(5) implies the existence
of $k\ge0$ and $l\ge0$ such that
\[\hat G/\tor(\hat G)\cong \Q_p^k\oplus \Z_p^l.\]
Duality applies and yields topological isomorphisms 
\[D\defi \tor(\hat G)^\perp\cong (\hat G/\tor(\hat G))^{\hat{\phantom{m}}}
\cong \Q_p^k\oplus \prf p^l.\]
Hence $D$ is a closed divisible subgroup  of $G$ 
having finite \prank, with compact factor group $G/D$.
\end{proof}

\begin{lemma}\label{l:U-fg-stqh}
Let a \lca\ $p$-group $G$ contain a \fg\ open subgroup $U$.
Then $G$ is \stqh.
\end{lemma}

\begin{proof} Since $G$ is periodic the \fg\ subgroup
$U$ is compact. 
Making use of Proposition \ref{p:finrank} 
and Theorem \ref{t:M-abelian-p}(a) 
we can pass to a smaller open subgroup of $U$ and achieve $U$
to be \tf. Moreover, since the smaller subgroup is open in
the \fg\ subgroup $U$ 
it is itself \fg\ (see e.g. \cite[Proposition 2.5.5]{ribes-zalesskii}). 
Then $\tor(G)$ turns out to be a discrete, hence closed
subgroup of $G$.
Let $X$ and $Y$ be closed subgroups of $G$. 
Since $U$ is compact and open, we can pass to a factor group of $G$
which still satisfies the premises of the lemma, and thereby modifying 
$X$ and $Y$ by making
 use of 
Lemma \ref{l:XYU} and achieve that $X$ and $Y$ are both discrete
subgroups of the modified group $G$. Therefore $X+Y$ is contained
in the discrete subgroup $\tor(G)$ and is thus closed in $G$.
\end{proof}

\begin{theorem}\label{th:tM=stqh p}
The following statements about a \lca\ $p$-group $G$ are equivalent:
\begin{enumerate}[(a)]
\item $G$ is \tM.
\item $G$ is \stqh.
\end{enumerate}
\end{theorem}

\begin{proof}
In light of Lemma \ref{l:stqh<tM}, we only need to show that (b)
is a consequence of (a).

Thus assume (a) and let $U$ be an open compact subgroup of $G$.
If $\rank_p(U)$ is finite then (b) is a consequence of Lemma \ref{l:U-fg-stqh}.
We assume from now on that the \prank\ of $U$ is infinite.

Let $X$ and $Y$ be closed subgroups of $G$. 
Making use of Lemma \ref{l:XYU} with $K\defi X\cap U+Y\cap U$
by modifying $X$ and $Y$ and replacing $G$ by $G/K$
(which is still \tM\ by Proposition \ref{p:stqh-class})
we  can arrange $X\cap U=Y\cap U=\{0\}$.
If $\rank_p(U)$ is finite then $G$ is \stqh\ by
Lemma \ref{l:U-fg-stqh}. 
Otherwise the \prank\ of $U$ is infinite and therefore
Theorem \ref{t:M-abelian-p}(b) implies that 
the discrete group $G/U$ is finite.
Moreover, the embeddings $X\to G/U$ and 
$Y\to G/U$ are embeddings of discrete
groups. Hence the \prank s of $X$ and $Y$ are finite and therefore
the \prank\ of $X+Y$ is finite (considered without topology).
Thus $X+Y$ has finite socle. 
Making our open compact subgroup $U$ small enough and observing
that the \prank\ of $U$ is then still infinite 
allows us to assume that $\socle(X+Y)$ and 
hence $X+Y$ intersect $U$ trivially.
Therefore $X+Y$ is a discrete subgroup of $G$ 
and hence is closed. 
Thus (b) holds.
\end{proof}

\begin{lemma}\label{l:prf} Let $G$ be a \lca\ $p$-group containing
a subgroup $D$ which algebraically is isomorphic to $\prf p^k$ for
some $k\ge1$. Then $D$ is a discrete and hence closed subgroup of $G$.
\end{lemma}
\begin{proof}
There is no loss of generality to assume $G=\overline D$. Let $U$ be
a compact open subgroup of $G$. We claim that $D\cap U$ must have finite
exponent. In order to see this we remark that $D\cap U$, as an abstract
group, has finite \prank\ at most $r$ and thus, algebraically
\[D\cap U\cong\prf p^l\oplus F\]
for some $0\le l\le r$ and finite $F$. Since $U$ is a compact $p$-group it is
reduced and cannot contain a divisible subgroup. Hence $l=0$ and hence
$D\cap U$ is finite. Since $D\cap U=F$ is dense in $U$ we may
conclude that $U$
itself is finite and thus $D$ must be a discrete subgroup of $G$.
\end{proof}

We are ready for proving the first main result.

\begin{proof}[Proof of Theorem \ref{t:mainA}]
Let (a) be true. Then (b) follows from Theorem \ref{t:M-abelian-p}.

\msk

Assume (b.1). This condition holds for closed subgroups and factor groups.
Let $X$ and $Y$ be any closed subgroups. Then, letting $U_X$ and $U_Y$
be open compact subgroups of respectively $X$ and $Y$,  and taking
Lemma \ref{l:XYU} into account, one may  pass to 
the factor group $H\defi G/K$ and hence assume that $X$ and $Y$ 
intersect $U$ trivially. Hence $X$ and $Y$ can be assumed to be discrete
and are therefore torsion subgroups of $H$. 
Since $\tor(H)$ is discrete conclude that 
$X+Y$ is discrete and hence closed.

\ssk

Assume (b.2). 
Let $X$ and $Y$ be any closed subgroups and, similarly as before, 
setting $U_X\defi U\cap X$ and $U_Y\defi U\cap Y$, and, using
Lemma \ref{l:XYU}, replace $X$ and
$Y$ respectively by $X+U_Y$ and $Y+U_X$, and,
factor $U_X+U_Y$ in $G$, we can achieve that $X$ and $Y$ can
be assumed to be discrete subgroups of $G$. 
If $U/(U_X+U_Y)$ has finite \prank, we may
apply the reasoning of case (b.1). Let us assume now that after factoring
$U_X+U_Y$ that $U/(U_X+U_Y)$ still has infinite \prank. Simplifying notation
we let $G$ and $U$ and $X$ and $Y$ denote the respective factor groups.
Since $(X+U)/U\cong X/(X\cap U)\cong X$ is discrete and has finite \prank\ by
condition (b.2) and a similar statement holds for $Y$, we can assume
\[ X=D_X+F_X, \ \ Y=D_Y+F_Y\]
for finite groups $F_X$ and $F_Y$ and finite \prank\ torsion divisible 
subgroups $D_X$ and $D_Y$ of $G$. Since
\[ X+Y=(D_X+D_Y)+(F_X+F_Y)\]
it will suffice to prove
that $\Delta\defi D_X+D_Y$ is closed and to note
that adding the finite summand $F_X+F_Y$ renders again a closed subgroup
of $G$. Since $\Delta$, as an abstract group,
is isomorphic to $\prf p^k$ for some $k\ge1$ it follows from Lemma \ref{l:prf}
that $\Delta$ and hence $X+Y$ is closed.

\ssk

Thus (b) implies (c).

\msk
 
That (c) implies (a) has been shown in Theorem \ref{th:tM=stqh p}.
\end{proof}

More can be said if $G$ is torsion.

\begin{corollary}\label{c:M-abelian-p}
Let $G$ be a \lca\ nondiscrete torsion $p$-group. 
Then $G$ is \tM\ if, and only if,
the maximal divisible subgroup $D\defi\div(G)$ is discrete and 
has finite \prank\ and
there is a compact open, and hence reduced, subgroup $R$ such that 
\[G=R\oplus D\]
algebraically and topologically. 
\end{corollary}

\begin{proof}
We discuss the cases (a) and (b) in Theorem \ref{t:M-abelian-p}.
If the premise in (a) holds then, since $U$ is a compact torsion group
having finite exponent and finite \prank, it is finite. Then 
$G$ would have to be discrete, contrary to our assumptions.
Thus $G$ satisfies premise (b) in Theorem \ref{t:M-abelian-p}. 
Therefore the \prank\ of $G/U$ is finite
and $D$ is a finite \prank\ divisible subgroup -- hence $D$ is a 
discrete subgroup of $G$. Passing to a smaller open
compact subgroup of $U$ if necessary,
\cite[Proposition 2.5.5]{ribes-zalesskii}
ensures finite generation, and, 
taking Proposition \ref{p:finrank}
into account, we can in addition assume $D\cap U=\{0\}$. 
Using a result of R.~Baer (see \cite[Theorem 22.2]{fuchs1}), one can
find a reduced subgroup $R$ of $G$ containing the reduced subgroup
$U$ with $G=D\oplus R$. Necessarily $R$ is open in $G$ and $G=D\oplus R$
algebraically and topologically. 
Hence $R$ is itself
compact. 

Since $R$ is open and $D\cap R=\{0\}$ we infer that
$D$ is discrete.
\end{proof}

\subsection{Totally Disconnected LCA-Groups}
We treat the case when $G$ is periodic first and later turn to 
totally disconnected but not periodic $G$.

\begin{proof}[Proof of Theorem \ref{t:abelian-tor-stqh}]
(a)\implies(b).
Assume first that $A$ is \tM.  
Select an open compact subgroup, say $U$, of $A$.
Since $A$ is torsion, $U$ is a compact abelian torsion group,
and therefore the set $\phi\defi \pi(U)$ must be finite.
Put $\delta\defi\pi(A)\setminus\phi$. 
From $A_\delta\cap U=\{0\}$ it follows that $A_\delta$ is a discrete
subgroup of $A$, proving (b.2). 

Next observe that 
\[A_\phi=\bigoplus_{p\in\phi}A_p\]
is a direct sum since $\phi$ is finite. 
Corollary \ref{c:M-abelian-p} 
(in conjunction with Theorem \ref{t:M-abelian-p})
 implies that for each $p$ in $\phi$ there
is a decomposition
\[A_p=D_p\oplus V_p,\]
with $D_p$ a divisible finite $p$-rank subgroup of $A$ and $V_p$ compact.
Thus there is a decomposition
\[A_\phi=D_\phi\oplus V_\phi\]
where $D_\phi\defi\bigoplus_{p\in\phi}D_p$ 
is a discrete divisible subgroup, 
and,
$V_\phi\defi\bigoplus_{p\in\phi}V_p$ is compact.
Thus also (b.1) is established.

\msk

(b)\implies(c). Assume (b) to hold.
Apply Lemma \ref{l:coprime}
to $A_\delta$ and the finitely many factors $A_p$ for $p\in\phi$.
Then $A_\delta$ being discrete, is \stqh\ 
and by Corollary \ref{c:M-abelian-p} (in conjunction with Theorem 
\ref{t:M-abelian-p}) 
so is $A_p$ for every $p$ in $\phi$.

\msk 

(c)\implies(a). 
Assume (c) to hold. 
Then, using Lemma \ref{l:stqh<tM}, (a) follows.
\end{proof}
\begin{proof}[Proof of Theorem \ref{t:M-periodic}].
If $G$ is \tM, then, for every $p\in\pi(G)$, the $p$-component
is a factor group of $G$ and hence, taking Remark \ref{rem:tMg}
into account, is \tM.

Let us sketch how to prove the converse, for more details see 
\cite{muk2}. 
For any closed subgroup $H$ of $G=\locprod_{p\in\pi}(G_p,C_p)$ one has
\[H=\gen{H_p:p\in\pi}=\locprod_{p\in\pi}(H_p,H_p\cap C_p).\]
It then follows that
\[X\vee Y=\gen{X_p\vee Y_p:p\in\pi} \ \ {\rm and} \ \
  X\wedge Y=\gen{X_p\wedge Y_p:p\in\pi}.\]
Using these equalities and the fact that $X\le Z$ if, and only if,
$X_p\le Z_p$ holds for all $p\in\pi$, one derives from
\[\forall p\in\pi: \ \ (X_p\vee Y_p)\wedge Z_p=X_p\vee(Y_p\wedge Z_p)\]
that $G$ is \tM.
\end{proof}

\begin{proof}[Proof of Theorem \ref{t:mainB}]
Remark first that $G$ is totally disconnected and \lca. Therefore
a compact open subgroup $U$ exists. Since $G$ is neither discrete
nor periodic $U$ is infinite.

(a)\implies(b): 

\ssk

(b.1)
Since $G$ is not periodic $\comp(G)$ is a proper subgroup of $G$.
Therefore Lemma \ref{l:not-tM}(b) shows that $\tor(G)=\comp(G)$.
By Lemma \ref{l:not-tM}(c),
the $\Z$-rank of $G/U$ is finite. Since $U\le \comp(G)$ both subgroups
are open and hence $G/U\to G/\comp(G)$ 
is a quotient map of discrete groups.
This implies that the $\Z$-rank of $G/\tor(G)$ 
does not exceed the $\Z$-rank
of $G/U$ and is therefore finite. 

\ssk

(b.2) 

Since, by (b.1), $T=\tor(G)$ is a \tM\ torsion group it follows from
Theorem \ref{t:abelian-tor-stqh} that $T$ is \stqh. 

Let us prove the extra statement 
about $G/N$ with $N$ a closed subgroup of $T$.
Certainly $\tor(G/N)=T/N$ because $N\le T=\tor(G)$. 
Since $N$ is a closed subgroup of $G$ it follows that $G/N$ is \tM.
As $N\le T$ it follows that $G/N$ cannot be periodic.
Indeed,
by the second isomorphism theorem $G/T\cong (G/N)/(T/N)$
must be torsion free of finite $\Z$-rank and
$\tor(G/N)=T/N$ is \stqh\ by Proposition \ref{p:stqh-class}. 
Thus the factor group $G/N$ will enjoy all the
properties listed in (b). 

\msk

(b) $\implies$ (c): 

Fix closed subgroups $X$ and $Y$ of $G$. During the proof 
we shall modify $X$ and $Y$ and factor
some closed subgroup $N$ of $\tor(G)$ with $N\le X\cap Y$. 
By the additional statement in (b) $G/N$ still enjoys
the properties in (b) and thus $X+Y$ will be closed if, and only if, 
$(X+Y)/N$ is closed in $G/N$, by Lemma \ref{l:quotient-closed}. 

Thus we need to show that $X+Y$
is closed and  first consider the cases:

\noindent $\alpha.)$ $X$ and $Y$ are both torsion.\msk

\noindent $\beta.)$ $Y$ is torsion.\msk

$\alpha.)$ Since $X+Y\le T=\tor(G)$ and the latter is \stqh\ by
(b), conclude that $X+Y$ must be closed.
Since $\phi$ is finite conclude
from $X+Y=\bigoplus_{p\in\phi}(X_p+Y_p)$ being topologically
isomorphic to the direct product of the $p$-primary groups
$(X_p+Y_p)$ that $X+Y$ is closed.

\msk

$\beta.)$
We already know that 
$\tor(X)+Y=\tor(X)+\tor(Y)$ is closed,
and we first remark that $X+Y$ is closed, 
if and only if, $(X+(\tor(X)+Y))/\tor(X)$   
and $\tor(X)$ are 
closed, in light of Lemma \ref{l:quotient-closed}. 
Therefore observing that 
$X/\tor(X)=X/\tor(G)\cap X\cong (X+\tor(G))/\tor(G)$ is discrete
and of finite $\Z$-rank, we may factor
the closed subgroup $\tor(X)$ and hence assume that $X$ is \tf, has 
finite $\Z$-rank, and, that $Y$ is torsion. Moreover, due to the
additional statement in (b), 
our modified group $G$ still satisfies (b). 
Since $\tor(G)$ is open in $G$ and $X\cap \tor(G)=\{0\}$ we have that
$M\defi X+\tor(G)$ is open and a direct sum $M=X\oplus\tor(G)$. Since
$Y$ is a closed subgroup of $\tor(G)$ deduce that $X\oplus Y$ is
closed in $M$ and hence in $G$. 

\msk

For finishing the proof of ``(b)\implies(c)''
let $X$ and $Y$ now be arbitrary closed subgroups of $G$. Then
$X+\tor(Y)$ and $Y+\tor(X)$ are closed subgroups of $G$ and
$X+Y$ is closed if, and only if, $(X+\tor(Y))+(Y+\tor(X))$ is closed.
Since 
\[\tor(X+\tor(Y))=\tor(Y+\tor(X))=\tor(X)+\tor(Y)\]
we may factor $\tor(X)+\tor(Y)$ and, taking the additional statement
in (b) into account,
in the sequel assume that
both, $X$ and $Y$, are \tf\ and 
hence discrete subgroups with finite $\Z$-rank. 
Therefore, if $r$ is the sum of the $\Z$-ranks of $X$ and $Y$, 
it turns out that every {\em algebraically} 
 \fg\ subgroup of $X+Y$ 
can be generated by at most $r$ elements. This property holds in particular
for $C\defi(X+Y)\cap U$ for $U$ any open compact subgroup of $G$. 
Since $U\le\tor(G)$ it is a compact torsion group and has therefore
finite exponent; and so has its subgroup $C$. 
Therefore $C$ is finite. Hence, passing to
a smaller open compact subgroup $V$ of $U$, 
one can arrange $(X+Y)\cap V=\{0\}$.
Therefore $X+Y$ is a discrete and hence closed subgroup of $G$.

\msk

Certainly (c) implies (a), by Lemma \ref{l:stqh<tM}.  
\end{proof}

\begin{remark}\label{r:nonsplit}\rm
The group $G$ in Theorem \ref{t:mainB}
need not be a split extension of $\tor(G)$ by $L$  -- 
even if $G$ is discrete,
as has been demonstrated by \cite[Appendix 1, Theorem A1.32]{hofmor}.
\end{remark}

\begin{corollary}\label{c:M-td}
Let the totally disconnected  nonperiodic 
 \lca\ \stqh\ group $G$. Then
every \tf\ subgroup is discrete and hence closed. Moreover,
every algebraically \fg\ subgroup is discrete and hence closed.
\end{corollary}

\begin{proof}
Since $G$ is not periodic the subgroup $\comp(G)$ is open and by the premises
it agrees with $T=\tor(G)$. Therefore, for $H$ any \tf\ subgroup of $G$,
one has $T\cap H=\{0\}$ showing that $H$ is a discrete subgroup of $G$.
The second statement follows from the first one and the
structure of algebraically  
finitely generated abelian groups. 
\end{proof}

\subsection{Groups  with a  Nontrivial Connected Component}

Via Pontryagin duality Theorem \ref{t:mainB}
implies at once the following structure theorem. 

\begin{theorem}\label{t:M-conn} 
The following statements about a \lca\ group $G$, neither compact
nor discrete, and with nontrivial component $G_0$, are equivalent:
\begin{enumerate}[\rm(a)]
\item $G$ is \tM.
\item There are a finite set of primes $\phi$ and a disjoint 
  set $\delta$ of primes and all of the following statements hold:
  \begin{enumerate}[\rm({b.}1)]
  \item The component $G_0$ is a finite dimensional compact connected
        subgroup of $G$.
  \item $G/G_0$ algebraically and topologically decomposes as
   \[G/G_0=F_\phi\oplus Z_\phi\oplus Z_\delta\oplus S_\delta,\]
   where $F_\phi$ is discrete of finite exponent, 
   $Z_\phi=\prod_{p\in\phi}Z_\phi$ is \tf\
   and for every $p\in\phi$ one has $\rank_p(Z_p)$ finite. Moreover,
   $Z_\delta=\prod_{p\in\delta}Z_p$ is compact and \tf, 
   and, for every $p\in\delta$, the $p$-primary subgroup 
   $Z_p\cong \Z_p^{\mathfrak m_p}$, where $\mathfrak m_p$ is some cardinal.
   The subgroup  $S_\delta$ is compact and coreduced (i.e., its dual is
   reduced).
  \item The preimage, say $K$, of $Z_\phi\oplus Z_\delta$ under the
   canonical epimorphism from $G$ onto $G/G_0$
   is a split extension of $G_0$ by $Z_\phi\oplus Z_\delta$,
   i.e.,  algebraically and topologically 
    \[K\cong G_0\oplus Z_\phi\oplus Z_\delta\]
   and the factor group $G/K$ is algebraically and topologically
   isomorphic to $F_\phi\oplus S_\delta$.
  \end{enumerate}
\item $G$ is \stqh.
\end{enumerate}
\end{theorem}

\begin{proof}
Suppose (a). 

\msk

{\em 
Claim 1: The dual $\hat G$ satisfies the premise of Theorem \ref{t:mainB}.}

\msk

Remark \ref{rem:tMg} implies that $\hat G$ is \tM. 
Since $G$ is neither discrete nor compact so is $\hat G$ by duality.
Therefore, in particular, $\comp(\hat G)<\hat G$. 
Now Lemma \ref{l:not-tM}(a) and (b) together
imply that $(\hat G)_0$ is compact
and $\comp(\hat G)=\tor(\hat G)$. Then $(\hat G)_0$ is a compact 
torsion group and therefore is trivial 
(see \cite[Corollary 8.5(a)(e)]{hofmor}). 
Hence $\hat G$ is totally disconnected and Claim 1 is established.

\msk

{\em Claim 2: {\rm (b.1)} holds.}

\msk

By \cite[Corollary 7.69]{hofmor} we have that 
$\widehat{G_0}\cong \hat G/\tor(\hat G)$.
The last statement in Theorem \ref{t:mainB}(b.1) shows that $\widehat{G_0}$
has finite $\Z$-rank. Therefore
 \cite[Theorem 8.22]{hofmor} implies that $G_0$
has finite dimension. Whence Claim 2 follows.

\msk

{\em Claim 3: {\rm (b.2)} holds.}

\msk 

Claim 1 shows that $\tor(\hat G)$ satisfies   
Theorem \ref{t:abelian-tor-stqh}(a). 
Therefore Theorem \ref{t:abelian-tor-stqh}(b.2)  
applied to $\tor(\hat G)$
yields a finite set $\phi$ of primes and
\[\tor(\hat G)=D_\phi\oplus E\oplus D_\delta \oplus R_\delta\] 
where $D_\phi$ is discrete, divisible, and of finite \prank\ for all
$p\in\phi$, $E$ is a compact subgroup and $\pi(E)\subseteq\phi$,
$D_\delta$ is a discrete divisible torsion group, and,
$R_\delta$ is reduced discrete subgroup. Moreover 
$\pi(D_\delta)\cup\pi(R_\delta)$ intersects $\phi$ trivially.

Dualization yields a decomposition

\[G/G_0\cong \tor(\hat G)\dach\cong Z_\phi\oplus F_\phi\oplus
Z_\delta\oplus S_\delta\]
where, for respective \psyl p subgroups $Z_p\cong \Z_p^{\mathfrak m_p}$ for
cardinalities $\mathfrak m_p$ and $p\in\phi\cup\delta$, 
we have $Z_\phi=\prod_{p\in\phi}Z_p$ and $Z_\delta=\prod_{p\in\delta}Z_p$. 
Moreover, for $p\in\phi$ the cardinality
$\mathfrak m_p$ is finite. Furthermore, $F_\phi\cong \widehat{E}$
is discrete and has finite exponent and $S_\delta\cong \widehat{R_\delta}$
is compact  and since $R_\delta$ is reduced $S_\delta$ is coreduced.

Thus Claim 2 is established.

\msk

Let us show that (b.3) holds. Note that $K$ is compact and
$K/G_0\cong Z_\phi\oplus Z_\delta$ is \tf\ and compact.
Therefore \cite[(25.30)(b)]{hewitt-ross1-book} 
implies a splitting, i.e., algebraically
and topologically
\[K\cong G_0\oplus Z_\phi\oplus Z_\delta.\]
Hence (b.3) holds.

\msk

(b) $\implies$ (c). We start with a simple observation: 

\ssk

{\em Claim 1: Let $K$ be any closed subgroup 
of $G_0$. Then the factor group $G_0/K$ is finite dimensional and 
$G/G_0\cong (G/K)/(G_0/K)$ is a topological isomorphism.}

\ssk

As $G_0$ has finite dimension so has $G_0/K$.
The second isomorphism theorem yields the second statement of the Claim.

\ssk  

{\em Claim 2: We may assume $G_0\cap X=G_0\cap Y=\{0\}$.}

\ssk

We may use Lemma \ref{l:XYU} with $G_0$ playing the role of $U$, modify
$X$ and $Y$ according to the lemma, make use of Claim 1 where
we let $K\defi X\cap G_0+Y\cap G_0$, pass to the factor group $G/K$.
Then certainly $X\cap G_0=Y\cap G_0=\{0\}$.
Observe that we still have a decomposition of $G/G_0$ as in (b.2).
Claim 2 holds.

\ssk

{\em Claim 3: 
Let $C$ be any closed subgroup $G$ with $C\cap G_0=\{0\}$
and let $\til C\defi (C+G_0)/G_0$.
Let $p:G\to G/G_0$ be the canonical projection and
$\phi_C:\til C\to C$ the canonical algebraic isomorphism.
All of the following holds:
\begin{enumerate}[\rm (i)]
\item
 For any closed subgroup $\til L$ of $\til C$ one has
 $\phi_C(\til L)=p^{-1}(\til L)\cap C$.
 Moreover, if $\til L$ is compact, then so is $\phi_C(\til L)$.
\item
 For $\sigma\subseteq \pi(\til C)$
 and $\til C_\sigma$ the $\sigma$-primary
 subgroup of $\til C$ one has $C_\sigma=\phi_C(\til C_\sigma)$.
\item
 There is an algebraic and topological direct decomposition 
 $\til C=\til Z_C\oplus \til F_C\oplus\til C_\delta$
 with $\til F_C=\tor{\til C_\phi}$ and $\til Z_C$ a \fg\ \tf\ 
 subgroup of $\til C_\phi$.
\item If $C$ is torsion and only contains $\phi$-elements, it is discrete
and has finite exponent.
\item
  The map $\phi_C$ induces algebraic and topological isomorphisms
\[Z_C\defi \phi_C(\til Z_C)\cong \til Z_C, \ \
  F_C \defi \tor(C_\phi) \cong \til F_C \ \ 
  {\rm and} \ \ C_\delta=\phi_C(\til C_\delta)\cong\til C_\delta\]
  and the subgroups $Z_C$  and $C_\delta$ are compact.
\item
  $C\cong\til C$ algebraically and topologically.
\end{enumerate}
}

\ssk

(i)
Let $i_C:C\to G$ be the canonical embedding and $j_C:C\to\til C$ be the
inverse of $\phi_C$. Then, as $j_C=p\circ i_C$, the map $j_C$
is continuous and one finds
\[\phi_C(\til L)=j_C^{-1}(\til L)
=i_C^{-1}p^{-1}(\til L)
=p_C^{-1}(\til L)\cap C.\]
Suppose that in addition $\til L$ is compact. Then, applying
Lemma \ref{l:XK} to $p_C^{-1}(\til L)$ where $G_0$ plays the
role of $K$ we find that $p_C^{-1}(\til L)$ is compact. Therefore
so is the intersection 
$\phi_C(\til L)=p_C^{-1}(\til L)\cap C$.

(ii)
Pick $x\in C_\sigma$. The algebraic isomorphism $\phi_C:\til C\to C$
induces a topological isomorphism 
\[\gen x\cong (\gen x+G_0)/G_0=\gen{\phi_C(x)}\le C.\]
Thus $\phi_C(\til C_\sigma)\le C_\sigma$. On the other hand, since
$\phi_C:\til C\to C$ is an algebraic isomorphism with continuous
inverse $j_C$, it follows that $j_C(C_\sigma)\le \til C_\sigma$,
and applying $\phi_C$ on both sides renders $C_\sigma\le \til C_\sigma$.
Therefore $\phi_C(\til X_\sigma)=C_\sigma$. 

\ssk

(iii)
The decomposition of $G/G_0$ in (b.2)
in conjunction with Lemma \ref{l:phi-kappa}
applied to the $\phi$-component of 
$\til C$ implies that 
\[\til C=\til Z_C\oplus \til F_C\oplus \til C_\delta\]
for $\til Z_C$ \fg\ \tf\ 
with $\pi(\til Z_C)\subseteq\phi$,
$\til F_C$ a discrete subgroup of finite exponent, and,
$\til Z_C$ and $\til C_\delta$ profinite groups 
with $\pi(\til X_\delta)\subseteq\delta$.

\ssk

(iv) Observe first that $\til C:=(C+G_0)/G_0$ is discrete
being contained in the discrete subgroup $\tor(G/G_0)_\phi$.
Thus there is an open compact subgroup $U$ of $G/G_0$ with
$\til C\cap U=\{0\}$. Let $p:G\to G/G_0$ be the canonical
projection. Then $W\defi p^{-1}(U)$ is open and 
\[ W\cap C\le (W\cap (C+G_0))\cap C\le G_0\cap C=\{0\}\]
showing that $C$ is indeed a discrete subgroup of $G$. 
Since $\tor(G)_\phi$ has finite exponent so has $\til C$
and thus also $C$.

\ssk

(v) 
Setting $\sigma\defi\phi$ in (ii) 
implies $C_\phi=\phi_C(\til C_\phi)$ and using (i) with $\til L\defi 
\til Z_C$ one obtains a topological isomorphism
$Z_C\defi\phi_C(\til Z_C)\cong\til Z_C$
since $\til Z_C$ and hence $Z_C$ are \fg\ and hence compact.

For proving the second equation we note that by (iv) 
$\tor(C_\phi)$ is discrete and hence

\[F_C\defi\phi_C(\til F_\phi)=\phi_C(\tor(\til C_\phi))=\tor(C_\phi).\]
The third topological isomorphism follows by letting $\sigma\defi\delta$ and
making use of (i) in order to see that $C_\delta$ is compact.

\ssk

(vi)
Since $G_0\cap C=\{0\}$ so that $C/G_0\cap C=C$, 
\cite[(5.32) Theorem]{hewitt-ross1-book} implies that
the map $\phi_C:\til C=(C+G_0)/G_0\to C$ 
is an algebraic isomorphism carrying open sets
to open sets. The compact subgroup $\til Z_C\oplus \til C_\delta$
is open in $\til C$ and therefore, taking (i) into account, its image 
\[\phi_C(\til Z_C\oplus\til C_\delta)=Z_C\oplus C_\delta\] 
is an open compact subgroup of $C$.
Thus the restriction of $\phi_C$ to $\til Z_C\oplus\til C$ is a
topological isomorphism onto an open subgroup of $C$ and therefore
$\phi_C$ is a topological isomorphism. Thus (vi) and hence Claim 3 are
established.

\msk

Let us finish proving ``(b)\implies(c)''.
Applying Claim 3(vi) to $X$ and $Y$ separately one finds
compact groups $Z_X$, $Z_Y$, $X_\delta$, $Y_\delta$ and
discrete torsion subgroups 
$F_X$ and $F_Y$ with $\pi(F_X)\cup\pi(F_Y)\subseteq\phi$
such that algebraically and topologically
\[X=Z_X\oplus F_X\oplus X_\delta, \ \ Y=Z_Y\oplus F_Y\oplus Y_\delta.\] 

The sum of compact subgroups
\[V:=Z_X+Z_Y+X_\delta+Y_\delta\]
is compact. Therefore 
\[X+Y=(F_X+F_Y)+V\]
will turn out to be closed if we can show that $F_X+F_Y$ is closed.

From (iv) it follows that $F_X$ and $F_Y$ are discrete torsion
groups of finite exponent and hence
closed subgroups of $G$ and Corollary \ref{c:findim-neu}
implies the finiteness of $(F_X+F_Y)\cap G_0$. 
Now we may use Claim 1 with $K\defi (F_X+F_Y)\cap G_0$ and 
assume $(F_X+F_Y)\cap G_0=\{0\}$. Since $\pi(F_X+F_Y)\subseteq\phi$ and
$F_X+F_Y$ is torsion deduce from (iv) that indeed $F_X+F_Y$ is closed.

Therefore (c) holds.

\msk 

(c) $\implies$ (a).
This follows from Lemma \ref{l:stqh<tM}.
\end{proof}

The preceding result corrects
Theorem \cite[Theorem 14.34(ii)]{HHR18}.
\section{Some Consequences}
\label{s:consequences}
From Corollary \ref{c:M-abelian-p} one obtains 
refined structure results for reduced, for torsion, and, for
divisible $p$-groups.

\begin{corollary}\label{c:reduced-p-stqh} 
Let $G$ be a reduced \lca\ torsion $p$-group. 
To be \stqh\ it is necessary and sufficient that
$G$ is either discrete or compact.
\end{corollary}

\begin{corollary}\label{c:stqh-p-divisible} 
Let $G$ be a \lca\ torsion \stqh\ $p$-group. Then either $G$ is discrete or
$\div(G)$ has finite $p$-rank.

In particular $\div(G)$ is a closed subgroup of $G$.
\end{corollary}

A referee provided the proof of the preceding proposition 
and actually proved the following very general result:

\begin{proposition}\label{p:referee} 
Let $G$ be a topological group with a subgroup topology
and $D$ a discrete divisible subgroup.
Then $G=D\oplus R$ algebraically and topologically 
for some open subgroup $R$.
\end{proposition}

\begin{proof}
Since $D$ is discrete there is an open subgroup $U$
of $G$ intersecting $D$ trivially and hence $D+U=D\oplus U$
algebraically and topologically.
The canonical epimorphism $\phi:D+U\to D$, by the universal
property of divisibility extends to a homomorphism $\phi':G\to D$.
The latter  agrees with $\phi$ on the open subgroup $D\oplus U$
of $G$ so that $\phi'$ is continuous. Therefore
\[G=D\oplus R\]
for $R\defi \ker(\phi')$.
\end{proof} 

For the $p$-group case the following immediate consequence will be helpful.

\begin{corollary}\label{c:stqh-ab-p-D-tor}
Let $G$ be a \stqhg\ having discrete maximal 
divisible subgroup $D\defi\div(G)$.
Then $G=D\oplus R$ for a reduced subgroup $R$
algebraically and topologically.
\end{corollary}

An additional consequence may be concluded.

\begin{corollary}\label{c:maxdiv}
Let $G$ be a \lca\ \stqh\ nondiscrete torsion $p$-group.
Then, for $C$ 
compact subgroup of $G$, 
\[(\div(G)+C)/C = \div(G/C).\]
\end{corollary}

\begin{proof}
Theorem \ref{th:tM=stqh p} implies that $G$ is \tM\ and Corollary 
\ref{c:M-abelian-p} shows that the maximal divisible subgroup
$D\defi\div(G)$ of $G$ is discrete and has finite \prank.
Moreover, $G=D\oplus R$, for some compact open torsion subgroup $R$ of $G$.
By Proposition \ref{p:cpt-tg} $R$ and hence $R/R\cap(D+C)$ have
finite exponent.
Because $(D+C)/C$ is divisible we have $(D+C)/C\le \div(G/C)$. 
In order to show that
$(D+C)/C$ is the {\em maximal} divisible subgroup of $G/C$ consider 
the algebraic isomorphisms 
\[(G/C)/((D+C)/C))\cong G/(D+C)=(R+D+C)/(D+C)\cong R/(R\cap(D+C)).\] 
It follows that $(G/C)/((D+C)/C)$ has finite exponent 
and is hence reduced 
showing the desired containment $\div(G/C)\le (D+C)/C$. 
\end{proof}

Every compact or discrete abelian $p$-group clearly is \stqh.
Therefore we first concentrate on groups neither 
compact nor discrete.

\begin{proposition}\label{p:p-red-stqh}
The following statements about a \lca\ reduced $p$-group $G$,
 neither discrete nor compact, are equivalent:
\begin{enumerate}[\rm (a)]
\item $G$ is \stqh.
\item $G$ contains an open \fg\ subgroup.
\end{enumerate}
\end{proposition}

\begin{proof}
Suppose (a). Then by Theorem \ref{th:tM=stqh p} $G$ is \tM\
and therefore either (a) or (b) of Theorem \ref{t:M-abelian-p} must hold.
In the latter case $G$ would have to be compact since $\div(D)=\{0\}$.
Hence case (a) of Theorem \ref{t:M-abelian-p} holds and therefore, 
as desired, $G$ has an open \fg\ subgroup.

\msk

That (b) implies (a) is an immediate 
consequence of Lemma \ref{l:U-fg-stqh}. 
\end{proof} 

\begin{remark}\label{r:p-red-stqh}\rm
An example of a reduced \lca\ $p$-group which is
\stqh\  and neither compact nor discrete can be
found in \cite[Remark 14.6]{HHR18}.
\end{remark}

\begin{lemma}\label{l:stqhg-p-DR} 
Let $G$ be a non-discrete \lca\ $p$-group with a \fg\
open subgroup $U$. Then the maximal
divisible subgroup $D$ is closed and $G$ algebraically
and topologically decomposes 
\[G=D\oplus R,\]
for $R$ a closed reduced subgroup of $G$.
\end{lemma}

\begin{proof}
Lemma \ref{l:U-fg-stqh} implies that $G$ is \stqh. 
Since $U$ is \fg\ Lemma \ref{l:Rb} implies that $U$ is
algebraically and topologically isomorphic to $V\oplus F$
for some finite subgroup $F$ and $V$ open and \tf.
If $V=\{0\}$ the group $G$ is discrete, a contradiction.
Thus $V\neq\{0\}$ and, replacing $U$ by $V$ we can arrange 
that $U$ is \tf. 
Then certainly $\overline{D\cap U}$ is \fg\ and hence by
Proposition 3.76 in \cite{HHR18} $D$ is closed and 
thus $X\defi D\cap U=\overline{D\cap U}$ is \fg.

Since $U/X$ is \fg, by Lemma \ref{l:Rb} 
there are a finite $p$-group $F$, some $r\ge0$, and,
a closed
subgroup $W$ of $U$ containing $X$ such that $U/X=F\oplus W/X$ and
$W/X\cong \Z_p^r$.
Lifting $r$ generators of $\Z_p^r$ to $W$ yields a closed subgroup 
$Z\cong\Z_p^r$ of $W$ such that 
\[W=X\oplus Z=(D\cap U)\oplus Z.\] 
The endomorphism $\eta:W\to D\cap U$ extends to a continuous
homomorphism $\til\eta:W+D\to D$ which restricts to the identity
on $D$. Therefore, setting $R\defi\ker(\til\eta)$  
\[G=D\oplus R\]
is a splitting. 
\end{proof}

Now we offer a new and short argument for the following
result of Mukhin \cite{muk2}.

\begin{proposition}\label{p:p-D-stqh}
A divisible \lca\ $p$-group $G$ is \stqh\ if and only if, for
some set $I$ and nonnegative integer $m$, 
\[G\cong \prf p^{(I)}\oplus\Q_p^m\]
algebraically and topologically.
\end{proposition}

\begin{proof}
Assume first that $G$ is \stqh. 
If $G$ is discrete it has the described structure for $m=0$.
Thus we may assume $G$ {\em not} to be discrete.
Fix an open compact subgroup $U$ of $G$.
Then $G/pU$ is a torsion \stqhg\
and Corollary \ref{c:M-abelian-p} 
(in conjunction with Theorem \ref{t:M-abelian-p}) 
implies that $G/pU$ is either
discrete or has finite $p$-rank. 
In either case $U/pU$ is finite and
hence $U$ has finite $p$-rank. Therefore,
by Proposition \ref{p:finrank}, 
\[U\cong F\oplus \Z_p^m\]
for some nonnegative integer $m$ and a finite subgroup $F$ of $G$.
Choosing the open compact subgroup $U$ small enough, we can arrange
\[U\cap \tor(G)=\{0\}.\]
Hence the subgroup $\tor(G)$ is discrete and divisible 
and is therefore, as a consequence of Lemma 4.4 in \cite{HHR-comfort},
topologically and algebraically a direct summand of $G$ intersecting
$U$ trivially. Then 
\[G=\tor(G)\oplus D_U,\]
with $D_U\cong \Q_p^m$ a divisible hull of $U$.

Since $\tor(G)$ is discrete and divisible there is a set $I$ with
\[\tor(G)\cong \prf p^{(I)}.\]

\msk

Conversely, suppose 
\[G\cong \prf p^{(I)}\oplus \Q_p^m\]
for $I$ some set and $m\ge0$. 
Since the \prank\ of the open summand
$\Q_p^m$ is finite (equal to $m$) 
there is an open \fg\ subgroup $U$ of $\Q_p^m$ and
hence of $G$. 
Lemma \ref{l:U-fg-stqh} implies that $G$ is \stqh.
\end{proof}

Next we turn to providing a full 
classification of the periodic \stqhg s, 
Theorem \ref{t:ab-periodic-stqh}.

\begin{definition}\label{d:img-per}\rm 
A periodic \lca\ group $G$ is {\em \im}, 
provided every \fg\ subgroup $H$ of $G$ can be 
topologically generated by a single element. 
\end{definition}

Note that a periodic \lca\ group $G$ is \im\ if, and only if,
for every $p\in\pi(G)$ the $p$-primary subgroup $G_p$ has \prank\ 1.
(see Subsection \ref{ss:Mg-p-case}). 

That \im\ groups are always \stqh\ could be read off 
from Mukhin's characterization 
of abelian \stqhg s, see \cite{muk2},  a very elementary proof 
of this fact follows.

\begin{proposition}\label{p:img-stqh}
Every periodic \im\ group $H$ is \stqh.
\end{proposition}

\begin{proof}
Assume first that $H$ is a $p$-group. Then, as a consequence of
Proposition \ref{p:finrank}, $H$ is isomorphic to either
the additive group of the $p$-adic field $\Q_p$, or to the additive group
of the $p$-adic integers $\Z_p$, or to a finite cyclic $p$-group,
or to Pr\"ufer's $p$-group $\prf p$. 
Then, for $X$ and $Y$ any closed subgroups of $H$, 
 either $X\subseteq Y$ or $Y\subseteq X$ must hold. 
 But then $X+Y$ agrees  with one of the closed subgroups $Y$ or  $X$
and is hence a closed subgroup of $H$.

\msk 

Let $H$ now be arbitrary and consider closed subgroups $X$ and $Y$.
Put $\pi\defi \{p\in \pi(H): X_p\subseteq Y_p\}$.
Then, letting $\pi'\defi \pi(X)\setminus\pi$,
there are direct sum decompositions
\[X=X_\pi\oplus X_{\pi'} \ \hbox{ and } \ Y=Y_\pi\oplus Y_{\pi'}.\]
For proving $X+Y$ to be a closed subgroup of $H$,
in light of Lemma \ref{l:coprime}, it suffices to prove closedness 
of the two subgroups $X_\pi+ Y_\pi$ and $X_{\pi'}+Y_{\pi'}$. 
In the first case the group in question agrees 
with the closed subgroup $Y_\pi$ and
in the second one with the closed subgroup $X_{\pi'}$. Hence $X+Y$ is a closed
subgroup. Thus $H$ is \stqh.
\end{proof}

Here is a complete description of all {\em torsion} \lca\ \stqhg s.

We can now complete Mukhin's classification
of abelian \stqhg s.
For a profinite abelian group $U$ the
{\em Frattini subgroup} $\Phi(U)$ is the 
intersection of all maximal open subgroups of $U$.
As pointed out in \cite{ribes-zalesskii} the Frattini subgroup 
of $G=\prod_pA_p$ takes the form
\[\Phi(A)=\prod_p\Phi(A_p)=\prod_ppA_p.\]

An elementary fact about \lca\ $p$-groups will be needed.

\begin{lemma}\label{l:ApMp}
Let $A$ be a \lca\ $p$-group containing properly 
an open compact subgroup $U$ of exponent $p$.
If $A$ is not \im\ then 
there are elements $a\in A\setminus U$ and 
$0\neq b\in U$ with $\gp a\cap \gp b=\{0\}$.
\end{lemma}

\begin{proof} 
Since the compact open subgroup $U$ has exponent $p$ 
the group $A$ is torsion and $U\le \socle(A)$. 
If there is $a\in\socle(A)\setminus U$
then $\gp a + U=\gp a\oplus U$ and we may pick $0\neq b\in U$ in order
to have $\gp a\cap \gp b=\{0\}$. 
Suppose next that $\socle(A)=U$. Pick any $a\in A\setminus U$. 
Since $A$, by assumption, is {\em not} \im\ and is torsion $U=\socle(A)$
cannot be cyclic and hence $U$ must contain 
a subgroup $L=\gp{x,y}\cong\Z(p)\oplus\Z(p)$. There must be 
$b\in L$ not belonging to $\socle(\gp a)$. Then $\gp a\cap \gp b=
\socle(\gp a)\cap \gp b=\{0\}$.
\end{proof}

We come to prove our 
addition to the classification results in \cite{muk2}.

\begin{proof}[Proof of Theorem \ref{t:ab-periodic-stqh}]
Assume (A). Then certainly $A_\delta\cap U=\{0\}$ 
and hence $A_\delta$ is a discrete \psyl\delta subgroup of $A$.
Since $A_\gamma\le U$ conclude that $A_\gamma$ is profinite and
so (ii) holds.
It follows that 
\[A=A_\delta\times A_\gamma\times A_{\delta'},\]
for $\delta'\defi \pi(A)\setminus\{\delta\cup\gamma\}$.
algebraically and topologically. Thus (i) is established. 

\msk

For establishing (iii) and (iv) 
 we may restrict ourselves
to the case $\delta=\gamma=\emptyset$, i.e., 
$A_\delta=A_\gamma=\{0\}$ from now on.

\msk

Since, by Proposition \ref{p:stqh-class}, 
every subgroup and quotient of a \stqhg\ again is a \stqhg, 
passing to subgroups of quotients of $A_\phi$
renders \stqhg s. For proving that $\phi$ must be finite we may 
factor $\Phi(U_\phi)=\prod_{p\in\phi}pU_p$ 
and achieve that $U_p$ has exponent $p$ only. By using 
Lemma \ref{l:ApMp} one can find for 
every $p\in\phi$ elements $a_p\in A_p\setminus U_p$ of order at most $p^2$
and $0\neq b_p\in U_p$ with $\gp{a_p}\cap \gp{b_p}=\{0\}$.
The closed subgroup 
\[L\defi \gen{a_p,b_p: p\in\phi}\]
of $A_\phi$ is still \stqh. Observe that 
$U\cap L=\gen{pa_p,b_p:p\in\phi}$ and, after factoring in it the
closed subgroup $N$ generated by all elements 
of the form $pa_p$ we find an algebraic and topological isomorphism
\[L/N\cong \bigoplus_{p\in\phi}\Z(p)\oplus \prod_{p\in\phi}\Z(p).\]
Since $L/N$ is \stqh\ deduce from Lem\-ma \ref{l:cyclic-stqh} 
that $\phi$ must be finite. Hence (iii) holds. 

\msk

(iv) is an immediate consequence of the fact 
that $p\in\mu$ if and only if $\rank_p(A_p)=1$ 
 if and only if $A_p$ is \im. 
 
 \msk
 
 Finally, $A$ certainly is the cartesian product of the Sylow subgroups $A_\delta$, $A_\phi$ and $A_\mu$.
 Hence (B) holds.
 
 \msk
 
Assume now (B). In light of Lemma \ref{l:coprime} 
it will suffice to prove that $A_\delta$, $A_\gamma$,
$A_\phi$, and $A_\mu$ are \stqh.
For  $A_\delta$ this is obvious since $A_\delta$ is discrete. 
Since $A_\gamma$ is compact it is \stqh.
For every $p$ in the finite set $\phi$ we know from (iii) that
$A_p$ is \stqh. Thus  applying Lemma \ref{l:coprime} 
to the finite product
\[A_\phi=\prod_{p\in\phi}A_p\]
shows that $A_\phi$ is \stqh.
Finally observe that $A_\mu$ has \prank\ $1$ \psyl p 
subgroups for all $p\in\mu$. 
\end{proof}

One may ask under which conditions on a \lca\ group $G$
the group is \tM\ if, and only if, $G$ is \stqh.
When $G$ is discrete then, as has been mentioned in the 
introduction, Dedekind \cite{Dedekind77}
proved that $G$, equipped with the discrete topology,
is \tM\ and it certainly is \stqh.
We conclude our work by contributing to the above question.

\begin{theorem}\label{t:M-stqh-td}
Let $G$ be a nonperiodic totally disconnected \lca\ group.
The following statements are equivalent:
\begin{enumerate}[\rm(a)]
\item $G$ is \tM.
\item $G$ is \stqh.
\end{enumerate}
\end{theorem}

\begin{proof}
Since, by Lemma \ref{l:stqh<tM}, (b) implies (a) 
we only need to deduce (b) from (a).

Thus assume that $G$ is \tM. 
If $G$ is compact then $G$ is certainly \stqh\ else
$G$ is neither discrete nor compact and the result
follows from Theorem \ref{t:mainB}. 
\end{proof}

Let us summarize the findings of Theorems \ref{th:tM=stqh p},
\ref{t:M-conn}, and, \ref{t:M-stqh-td} 
and keep in mind Lemma \ref{l:cyclic-stqh} and Remark \ref{r:tqh-stqh}:

\begin{corollary}\label{l:summary-M-stqh}
Under the following conditions a nondiscrete \lca\ group
$G$ is \tM\ if, and only if, $G$ is \stqh.
\begin{enumerate}[\rm(i)]
\item $G$ is a $p$-group.
\item $G$ is totally disconnected but not periodic.
\item $G_0$ is not trivial.
\end{enumerate}
Moreover, whenever $G$ satisfies one of the conditions (i)--(iii)
then the Pontryagin dual $\hat G$ is \stqh\ if, and only if,
$G$ is \stqh.
\end{corollary}

\ignore{
\subsection*{Acknowledgement} 
We are thankful to the referees, not only
for reading our paper carefully and proposing numerous valuable suggestions
which lead to significant improvement of our work, but particularly 
for continued interest in our results and 
going with enormous patience and endurance through
several gradually improved versions of our paper. }

\bibliographystyle{abbrv}
\bibliography{BOOK}
\end{document}